\documentclass[a4paper,reqno,10pt]{amsart}

\usepackage{amssymb}
\usepackage{amstext}
\usepackage{amsmath}
\usepackage{amscd}
\usepackage{amsthm}
\usepackage{amsfonts}
\usepackage{graphicx}
\usepackage{latexsym}
\usepackage{mathrsfs}
\usepackage{mathtools}
\usepackage[all,poly,necula]{xy}
\usepackage{ifthen}
\usepackage{enumitem}

\xyoption{all}

\usepackage{lscape}

\usepackage{ stmaryrd }
\usepackage{multirow}

\usepackage{tikz}
\usetikzlibrary{arrows,decorations.pathmorphing,decorations.pathreplacing}

\tikzset{vertex/.style={circle,fill=black,inner sep=1.3pt,outer sep=2pt},
         mvertex/.style={rectangle,draw=black,thick,inner sep=2pt,outer sep=2pt},
         cvertex/.style={circle,fill=white,draw=black,thick,inner sep=1.7pt,outer sep=2pt},
         cver/.style={circle,fill=white,draw=black,thick,inner sep=3pt,outer sep=2pt},
         rver/.style={rectangle,fill=white,draw=black,thick,inner sep=2pt,outer sep=2pt},
         tvertex/.style={inner sep=1pt,font=\upshape},
         unvertex/.style={circle,fill=white,draw=white,inner sep=1pt},
         fill1/.style={fill=black!20,draw=black!20},
         fill2/.style={fill=black!40,draw=black!40},
         fill12/.style={fill=black!60,draw=black!60},
         >=stealth',
         leadsto/.style={-angle 90,decorate,decoration=snake,very thick},
         cut/.style={decorate,decoration=saw,very thick}}

\newtheorem{theorem}{Theorem}[section]
\newtheorem{theoremi}{Theorem}

\newtheorem{propositioni}[theoremi]{Proposition}

\newtheorem{corollary}[theorem]{Corollary}
\newtheorem{lemma}[theorem]{Lemma}
\newtheorem{proposition}[theorem]{Proposition}
\newtheorem{definition-proposition}[theorem]{Definition-Proposition}

\theoremstyle{definition}
\newtheorem{definition}[theorem]{Definition}
\newtheorem{remark}[theorem]{Remark}
\newtheorem{example}[theorem]{Example}

\setlist[1]{label={\upshape(\arabic*)}}
\setlist[2]{label={\upshape(\alph*)}}

\renewcommand{\AA}{\mathcal{A}}
\newcommand{\CC}{\mathcal{C}}
\newcommand{\MM}{\mathcal{M}}
\newcommand{\DD}{\mathcal{D}}
\newcommand{\FF}{\mathcal{F}}
\newcommand{\PP}{\mathcal{P}}
\newcommand{\II}{\mathcal{I}}
\newcommand{\WW}{\mathcal{W}}
\newcommand{\XX}{\mathcal{X}}
\newcommand{\XXX}{\mathsf{X}}

\newcommand{\Z}{\mathbb{Z}}
\renewcommand{\P}{\mathbb{P}}

\newcommand{\Ext}{\operatorname{Ext}\nolimits}

\newcommand{\Hom}{\operatorname{Hom}\nolimits}

\newcommand{\End}{\operatorname{End}\nolimits}

\newcommand{\op}{\operatorname{op}\nolimits}

\newcommand{\Image}{\operatorname{Im}\nolimits}

\newcommand{\Cokernel}{\operatorname{Coker}\nolimits}

\newcommand{\Ab}{\mathcal{A}b}
\newcommand{\coker}{\Cokernel}
\newcommand{\im}{\Image}

\renewcommand{\u}{\underline}

\DeclareMathOperator{\moduleCategory}{\mathsf{mod}} \renewcommand{\mod}{\moduleCategory}
\DeclareMathOperator{\Mod}{\mathsf{Mod}}

\DeclareMathOperator{\proj}{\mathsf{proj}}

\DeclareMathOperator{\ind}{\mathsf{ind}}
\DeclareMathOperator{\Sub}{\mathsf{Sub}}

\DeclareMathOperator{\GP}{\mathsf{GP}}

\DeclareMathOperator{\add}{\mathsf{add}}

\DeclareMathOperator{\id}{\mathsf{id}}

\newcommand{\iso}{\cong}
\newcommand{\infl}{\rightarrowtail}
\newcommand{\defl}{\twoheadrightarrow}

\newcommand{\equi}{\simeq}

\renewcommand{\AA}{\mathcal{A}}
\newcommand{\BB}{\mathcal{B}}
\newcommand{\EE}{\mathcal{E}}

\numberwithin{equation}{section}

\xymatrixrowsep{1.5em}

\newdir{ >}{{}*!/-1em/@{>}}
\newcommand{\inflr}{\ar@{ >->}[r]}
\newcommand{\infld}{\ar@{ >->}[d]}
\newcommand{\deflr}{\ar@{->>}[r]}
\newcommand{\defld}{\ar@{->>}[d]}

\begin{document}
\title[Classifying exact categories via Wakamatsu tilting]{Classifying exact categories via Wakamatsu tilting}

\author[H. Enomoto]{Haruhisa Enomoto}
\address{Graduate School of Mathematics, Nagoya University, Chikusa-ku, Nagoya. 464-8602, Japan}
\email{m16009t@math.nagoya-u.ac.jp}

\begin{abstract}
 Using the Morita-type embedding, we show that any exact category with enough projectives has a realization as a (pre)resolving subcategory of a module category. When the exact category has enough injectives, the image of the embedding can be described in terms of Wakamatsu tilting (=semi-dualizing) subcategories. If moreover the exact category has higher kernels, then its image coincides with the category naturally associated with a cotilting subcategory up to summands.
 We apply these results to the representation theory of artin algebras. In particular, we show that the ideal quotient of a module category by a functorially finite subcategory closed under submodules is a torsionfree class of some module category.
\end{abstract}

\maketitle

\tableofcontents

\section{Introduction}
Since Quillen introduced exact categories in \cite{qu}, many branches of mathematics have made use of this concept.
One of the most important classes of exact categories is given by a cotilting module $U$ over a ring as the associated $\Ext$-orthogonal category $^\perp U$, which forms an exact category with enough projectives and injectives.
This class of exact categories is fundamental in the representation theory of algebras, such as the tilting theory in derived categories as well as Cohen-Macaulay representations.
In this paper, we give characterizations of such kinds of exact categories as $^\perp U$ among all exact categories by investigating the relationship between the following.
\begin{enumerate}
\item Categorical properties of exact categories, e.g. enough projectives, enough injectives, Frobenius, idempotent complete, having (higher) kernels, abelian, $\cdots$.
\item Representation theoretic realizations of exact categories, e.g. $\mod\Lambda$ for a ring $\Lambda$, $\Ext$-orthogonal category $^\perp U$ for a cotilting module $U$, the exact category $\XXX_W$ associated with a Wakamatsu tilting module $W$, their resolving subcategories, $\cdots$.
\end{enumerate}

As a consequence, we deduce several known results including \cite{chen,kal,kal2}.
Our approach is based on the Morita-type embedding: For a skeletally small exact category $\EE$ with enough projectives, the category $\PP$ of projective objects in $\EE$ gives a fully faithful exact functor (Proposition \ref{ff}) 
\[
\P: \EE \to \Mod\PP, \hspace{5mm} X \mapsto \EE(-,X)|_\PP.
\]
This gives a realization of $\EE$ as a subcategory of the module category.

First we apply this functor $\P$ to show the following basic observation, where the category $\mod \CC$ consists of finitely presented $\CC$-modules in a stronger sense, see Definition \ref{modc}.

\begin{propositioni}(=Proposition \ref{preresol})\label{propa}
Let $\EE$ be a skeletally small exact category. The following are equivalent.
\begin{enumerate}
\item $\EE$ is idempotent complete and has enough projectives.
\item There exists a skeletally small additive category $\CC$ such that $\EE$ is exact equivalent to some resolving subcategory of $\mod \CC$.
\end{enumerate}
\end{propositioni}

When $\EE$ has enough projectives and injectives, we show that the image of all injective objects under the above embedding $\P$ forms a special subcategory, which we call a \emph{Wakamatsu tilting subcategory}. This is a categorical analogue of a Wakamatsu tilting module introduced in \cite{wa} as a common generalization of a tilting module and a cotilting module. It is also known as a \emph{semi-dualizing module} \cite{ch,aty}, which is a certain analogue of a dualizing complex of Grothendieck \cite{residue}.
Any Wakamatsu tilting subcategory $\WW$ of a module category $\mod\CC$ gives rise to an exact subcategory $\XXX_\WW$ of $\mod\CC$ which has enough projectives and injectives, see Proposition \ref{xxww}. The category $\XXX_\WW$ for the special case $\WW = \proj \CC$ is nothing but the category $\GP \CC$ of Gorenstein projective $\CC$-modules (see Definition \ref{gpc}).

Using these concepts, one can embed exact categories with enough projectives and injectives as follows.

\begin{theoremi}(=Theorem \ref{pe})\label{theoremb}
Let $\EE$ be a skeletally small exact category. The following are equivalent.
\begin{enumerate}
\item $\EE$ is idempotent complete and has enough projectives and injectives.
\item There exist a skeletally small additive category $\CC$ and a Wakamatsu tilting subcategory $\WW$ of $\mod\CC$ such that $\EE$ is exact equivalent to some resolving-coresolving subcategory of $\XXX_\WW$.
\end{enumerate}
\end{theoremi}

As an application, we characterize when a given exact category is exact equivalent to one of the three important cases $\mod\CC$, $\XXX_\WW$ or $\GP \CC$ for some additive category $\CC$ and some Wakamatsu tilting subcategory $\WW$, see Theorem \ref{rec1}, \ref{rec2} and \ref{rec3} respectively.

Next we consider a special case when $\WW$ is a cotilting subcategory of $\mod\CC$. The category $\XXX_\WW$ coincides with the $\Ext$-orthogonal subcategory $^\perp\WW$ in this situation, and we give a simple characterization when $\EE$ and $\XXX_\WW$ are exact equivalent.
We extend the notion of \emph{$n$-kernels} \cite{jasso} to $n\geq -1$ (see Definition \ref{nkernel}), and prove the following main result.

\begin{theoremi}(=Theorem \ref{cor1})\label{theoremc}
Let $\EE$ be a skeletally small exact category and $n$ be an integer $n \geq 0$. The following are equivalent.
\begin{enumerate}
 \item $\EE$ is idempotent complete, has enough projectives and injectives and has $(n-1)$-kernels.
 \item There exist a skeletally small additive category $\CC$ with weak kernels and an $n$-cotilting subcategory $\WW$ of $\mod\CC$ such that $\EE$ is exact equivalent to $\XXX_\WW$.
\end{enumerate}
\end{theoremi}
This theorem provides us with a concrete method to prove that a given exact category is equivalent to those associated with a cotilting module. Using this, we prove some results on artin algebras over a commutative artinian ring $R$.

\begin{theoremi}(=Theorem \ref{main1})\label{theoremd}
Let $\Lambda$ be an artin $R$-algebra and $M \in \mod\Lambda$, and put $\CC:=\Sub M$. Consider the quotient category $\EE = (\mod\Lambda)/[\CC]$. Then the following hold.
\begin{enumerate}
\item $\EE$ admits an exact structure in which $\EE$ has enough projectives and injectives and has $0$-kernels.
\item Suppose that $\ind (\tau^- \CC) \setminus \ind \CC$ is a finite set. Then there exist an artin $R$-algebra $\Gamma$ and a $1$-cotilting module $U\in\mod\Gamma$ such that $\EE$ is exact equivalent to a torsionfree class $^\perp U \subset \mod\Gamma$.
\end{enumerate}
\end{theoremi}
A key ingredient of the proof of Theorem \ref{theoremd} is the theory of relative homological algebra due to Auslander-Solberg \cite{as1,as2,as3} and its relation to exact categories via \cite{drss}. 

As another application of Theorem \ref{theoremc}, we deduce the following result by Auslander-Solberg.
\begin{theoremi}(=Theorem \ref{main2})
Let $\Lambda$ be an artin $R$-algebra and $M\in\mod\Lambda$. Set $G := \Lambda \oplus M$, $C := D \Lambda \oplus \tau M$, $\Gamma := \End_\Lambda(G)$ and $U := \Hom_\Lambda(G, C) \in \mod\Gamma$. Then the following hold.
\begin{enumerate}
\item $U$ is a cotilting $\Gamma$-module with $\id U =2$ or $0$.
\item $\Hom_\Lambda(G,-):\mod\Lambda \to \mod\Gamma$ induces an equivalence $\mod\Lambda \equi {}^\perp U$.
\item $\mod \Lambda$ admits an exact structure such that projective objects are precisely objects in $\add G$ and the equivalence $\mod\Lambda \equi {}^\perp U$ is an exact equivalence.
\item $\End_\Lambda(G)$ and $\End_\Lambda(C)$ are derived equivalent.
\end{enumerate}
\end{theoremi}
This theorem is essentially contained in \cite[Proposition 3.26]{as2}, but we give a simple proof by using the modified exact structure on $\mod\Lambda$ and Theorem \ref{theoremc}.\\

Finally, let us give a brief description of the individual sections. 
In Section 2, we study exact categories with enough projectives. We prove Proposition \ref{propa} and give a characterization of exact categories of the form $\mod\CC$ for some additive category $\CC$.
In Section 3, we study exact categories with enough projectives and injectives. To this purpose, we introduce Wakamatsu tilting subcategories and prove some properties. Then we prove Theorem \ref{theoremb}, and give a characterization of exact categories coming from Wakamatsu tilting subcategories.
In Section 4, we introduce cotilting subcategories, and study their relationship to higher kernels.
In Section 5, we apply these results to the representation theory of artin algebras.
In Appendix A, we develop the analogue of Auslander-Buchweitz approximation theory in the context of exact categories, which we need in Section 4.
In Appendix B, we collect some results which enable us to construct new exact structures from a given one, which we use in Section 5.

\subsection{Notation and conventions}
Throughout this paper, \emph{we assume that all categories are skeletally small unless otherwise stated}, that is, equivalent to small categories. In addition, \emph{all subcategories are assumed to be full and closed under isomorphisms}. When we consider modules over rings or categories, we always mean right modules.
Let $\CC$ and $\CC'$ be categories, $\DD$ a subcategory of $\CC$ and $F:\CC \to \CC'$ a functor. We denote by $F(\DD)$ the essential image of $\DD$ under $F$, that is, the subcategory of $\CC'$ consisting of all objects isomorphic to the images of objects in $\CC$ under $F$. We simply call $F(\DD)$ the \emph{image} of $\DD$ under $F$. 

We assume that all functors between additive categories are additive.
For a subcategory $\DD$ of an additive category $\CC$, we write $\add\DD$ for the subcategory of $\CC$ consisting of all objects which are summands of finite direct sums of objects in $\DD$, that is, the smallest additive subcategory containing $\DD$ which is closed under summands.
We say that an additive category $\CC$ is \emph{of finite type} if there exist only finitely many indecomposable objects in $\CC$ up to isomorphism.

When we consider an exact(=extension-closed) subcategory $\CC$ of an exact category $\EE$, we always regard $\CC$ as the exact category with the exact structure inherited from those of $\EE$. 
As for exact categories, we use the terminologies \emph{inflations}, \emph{deflations} and \emph{conflations}. We denote by $\infl$ an inflation and by $\defl$ a deflation whenever we consider exact categories. 
We say that a functor $F:\EE \to \EE'$ between exact categories is an \emph{exact equivalence} if $F$ is an equivalence of categories, $F$ is exact and $F$ reflects exactness. This means that for a complex $A \to B \to C$ in $\EE$, it is a conflation in $\EE$ if and only if $F(A) \to F(B) \to F(C)$ is a conflation in $\EE'$.
For an exact category $\EE$, we denote by $\PP(\EE)$ (resp. $\II(\EE)$) the subcategory of $\EE$ consisting of all projective (resp. injective) objects. We say that \emph{$\EE$ has enough projectives $\PP$ (resp. enough injectives $\II$)} if $\EE$ has enough projective objects and $\PP = \PP(\EE)$ (resp. enough injective objects and $\II = \II(\EE)$).

We always denote by $R$ a commutative artinian ring. For an artin $R$-algebra $\Lambda$, we denote by $D:\mod\Lambda \to \mod\Lambda$ the standard Matlis duality and by $\tau$ and $\tau^-$ the Auslander-Reiten translations.

We refer the reader to \cite{ASS,ARS} for background on representation theory of algebras and to \cite{buhler,keller} for the basic concepts of exact categories.

\section{Exact categories with enough projectives}\label{2.1}
In this section, \emph{we always assume that $\EE$ is a skeletally small exact category with enough projectives $\PP =\PP(\EE)$.} We freely use basic properties of the bifunctor $\Ext^i_\EE(-,-)$ for $i\geq 0$, which can be defined by using projective resolutions. 

\subsection{Morita-type theorem}
We start with recalling the notion of modules over a category. For an additive category $\CC$, a \emph{right $\CC$-module} $X$ is a contravariant additive functor $X: \CC^{\op} \to \Ab$ from $\CC$ to the category of abelian groups $\Ab$. We denote by $\Mod\CC$ the (not skeletally small) category of right $\CC$-modules, and morphisms are natural transformations between them.
Then the category $\Mod \CC$ is an abelian category with enough projectives, and projective objects are precisely direct summands of (possibly infinite) direct sums of representable functors $\CC(-,C)$ for $C \in \CC$.

For a skeletally small exact category $\EE$ with enough projectives $\PP:=\PP(\EE)$, we define a functor
\begin{equation}\label{moritaembedding}
\P : \EE \to \Mod \PP, \hspace{5mm} X \mapsto \EE(-,X)|_\PP,
\end{equation}
which we call the \emph{Morita embedding} because of the following properties.
\begin{proposition}\label{ff}
Let $\EE$ be a skeletally small exact category with enough projectives $\PP$. Then the Morita embedding functor $\P:\EE \to \Mod \PP$ is fully faithful and exact. Moreover $\P$ preserves projectivity and all extension groups. 
\end{proposition}
\begin{proof}
This is well-known and standard. For the convenience of the reader, we shall give a proof.

We first observe that $\P$ is an exact functor. In fact, $\P$ is obviously left exact, and for each projective object $P$ and each deflation $f:Y \twoheadrightarrow Z$, the induced map $\EE(P,Y) \to \EE(P,Z)$ is surjective by the definition of projectivity. In addition, if $P\in\EE$ is projective, then $\P P = \PP(-,P)$ is a projective $\PP$-module. Therefore $\P$ sends projective objects to projective modules.

Next we will see that $\P$ is fully faithful.
Since $\EE$ has enough projectives, for any object $X \in \EE$, there exist conflations $X_2 \rightarrowtail P_1 \twoheadrightarrow X_1$ and $X_1 \rightarrowtail  P_0 \twoheadrightarrow X$. By the exactness of $\P$, this gives an exact sequence $\P P_1 \to \P P_0 \to \P X \to 0$ in $\Mod \PP$, and moreover we can check that $P_1 \to P_0 \to X$ is also a cokernel diagram in $\EE$. Hence we have the following diagram
\begin{eqnarray*}
\xymatrix{
0 \ar[r] & \EE(X,Y) \ar[d]^\P \ar[r] & \EE(P_0,Y)  \ar[d]^\P_\wr \ar[r] & \EE(P_1,Y) \ar[d]_\wr^\P \\
0 \ar[r] & (\Mod\PP)(\P X,\P Y) \ar[r] & (\Mod\PP)(\P P_0,\P Y) \ar[r] & (\Mod\PP)(\P P_1,\P Y) } 
\end{eqnarray*}
whose rows are exact and the second and third vertical morphisms are isomorphisms by the Yoneda lemma. It follows that $\EE(X,Y)\to (\Mod\PP)(\P X,\P Y)$ is also an isomorphism, thus $\P$ is fully faithful.

Note that $\P$ sends a projective resolution of $X$ to a projective resolution of $\P X$, since it preserves projectivity and exactness. Therefore $\P$ induces an isomorphism $\Ext_\EE^i(X,Y) \iso \Ext_{\Mod\PP}^i(X,Y)$ for all $i \geq 0$ because $\P$ is fully faithful. 
\end{proof}

However, $\P$ is far from essentially surjective because $\Mod\PP$ is too large. This leads us to the following definition of a subcategory of $\Mod \PP$. 
Suppose that $\CC$ is an additive category. Recall that a right $\CC$-module $M$ is called \emph{finitely generated} if there exists an epimorphism $\CC(-,C) \defl M$ for some $C\in\CC$. By the Yoneda lemma, finitely generated projective $\CC$-modules are precisely direct summands of representable functors. 

\begin{definition}\label{modc}
Let $\CC$ be a skeletally small additive category.
\begin{enumerate}
\item
We denote by $\proj\CC$ the full subcategory of $\Mod\CC$ consisting of all finitely generated projective $\CC$-modules.
\item
We denote by $\mod\CC$ the full subcategory of $\CC$-modules $X \in \Mod\CC$ such that there exists an exact sequence in $\Mod \CC$ of the form
\begin{eqnarray}\label{projeresol}
\cdots \to P_2 \to P_1 \to P_0 \to X \to 0
\end{eqnarray}
where $P_i$ is in $\proj \CC$ for each $i \geq 0$.
\end{enumerate}
In the same way, for a ring $\Lambda$ we define the subcategories $\proj\Lambda$ and $\mod\Lambda$ of $\Mod\Lambda$.
\end{definition}

Note that the notation $\mod\CC$ often stands for weaker notions, e.g. the category of finitely presented $\CC$-modules. In Proposition \ref{wk} below, we characterize when these notions coincide.

Let us recall the following classes of subcategories of exact categories.
\begin{definition}
Let $\EE$ be an exact category with enough projectives and $\CC$ a full subcategory of $\EE$ which is closed under isomorphisms. 
\begin{enumerate}
\item We say that $\CC$ is \emph{resolving} if $\CC$ satisfies the following conditions.
 \begin{enumerate}
 \item $\CC$ contains all projective objects in $\EE$.
 \item $\CC$ is closed under extensions, that is, for each conflation $L \infl M \defl N$, if both $L$ and $N$ are in $\CC$, then so is $M$.
 \item $\CC$ is closed under kernels of deflations, that is, for each conflation $L \infl M \defl N$, if both $M$ and $N$ are in $\CC$, then so is $L$.
 \item $\CC$ is closed under summands.
 \end{enumerate}
 Dually we define \emph{coresolving} subcategories for an exact category with enough injectives.
\item $\CC$ is called \emph{thick} if it is closed under extensions, kernels of deflations, cokernels of inflations and summands.
\end{enumerate}
\end{definition}
In addition to these classical concepts, we need the following weaker notion when we  later investigate exact categories which are not idempotent complete.
\begin{definition}\label{defpreresol}
Let $\EE$ be an exact category with enough projectives and $\CC$ a full subcategory of $\EE$ closed under isomorphisms. We call $\CC$ \emph{preresolving} if $\CC$ satisfies the following conditions.
 \begin{enumerate}
 \item[(a)] $\add\CC$ contains all projective objects in $\EE$.
 \item[(b)] $\CC$ is closed under extensions.
 \item[(c)] For each $X$ in $\CC$, there exists a conflation $\Omega X \infl P \defl X$ in $\EE$ such that all terms are in $\CC$ and $P$ is projective in $\EE$.
 \end{enumerate}
 Dually we define \emph{precoresolving} subcategories for an exact category with enough injectives.
\end{definition}

We point out that these subcategories are actually \emph{exact} subcategories of $\EE$, because they are closed under extensions.

The following lemma gives the relationship between resolving and preresolving subcategories.

\begin{lemma}\label{resol}
Let $\EE$ be an exact category with enough projectives $\PP$. Then a subcategory $\CC$ of $\EE$ is resolving if and only if it is preresolving and closed under summands.
\end{lemma}
\begin{proof}
The ``only if'' part is clear. Now we suppose that $\CC$ is preresolving and closed under summands. Since $\add \CC = \CC$ in this case, it suffices to show that $\CC$ is closed under kernels of deflations. Let $X\infl Y\defl Z$ be a conflation in $\EE$ in which $Y$ and $Z$ are in $\CC$. By our assumption for $\CC$, there exists a conflation $\Omega Z \infl P \defl Z$ with $\Omega Z$ in $\CC$ and $P$ in $\CC \cap \PP$. We then have the following pullback diagram.
\[
\xymatrix{
 & \Omega Z \ar@{=}[r] \infld & \Omega Z \infld \\
X \inflr \ar@{=}[d] & E \deflr \defld & P \defld \\
X \inflr & Y \deflr & Z
}
\]
Because $\CC$ is closed under extensions, the middle column shows that $E$ is in $\CC$. Furthermore the middle row splits because $P$ is projective in $\EE$. From this $X$ is a summand of $E \in \CC$, which shows that $X$ is actually in $\CC$ since $\CC$ is closed under summands.
\end{proof}

We give basic properties of $\mod\CC$.
\begin{proposition}\label{mod}
Let $\CC$ be a skeletally small additive category. Then $\mod \CC$ is a thick subcategory of $\Mod \CC$. In addition, $\mod\CC$ has enough projectives, and its projective objects are precisely objects in $\proj\CC$.
\end{proposition}
\begin{proof}
It follows from the horseshoe lemma that $\mod\CC$ is closed under extensions. First we will show that $\mod\CC$ is closed under cokernels of monomorphisms. Let $0\to X \to Y \to Z \to 0$ be an exact sequence in $\Mod\CC$ and suppose that $Y$ is in $\mod\CC$. By the definition of $\mod\CC$, there is an exact sequence $0\to \Omega Y \to P \to Y \to 0$ such that $P$ is in $\proj\CC$ and $\Omega Y$ is in $\mod\CC$. We then have the following commutative exact diagram
\begin{equation}\label{dekai}
\xymatrix{
  & 0 \ar[d] & 0 \ar[d] & & \\
  & \Omega Y \ar@{=}[r]\ar[d] & \Omega Y\ar[d] & & \\
0 \ar[r] & \Omega Z \ar[r] \ar[d] & P \ar[r] \ar[d] & Z \ar[r] \ar@{=}[d] & 0 \\
0 \ar[r] & X \ar[r]\ar[d] & Y \ar[r]\ar[d] & Z \ar[r] & 0 \\
 & 0 & 0 & &
}
\end{equation}
whose rows and columns are exact.
If $X$ is also in $\mod\CC$, we have that $\Omega Z$ is also in $\mod\CC$ since $\mod\CC$ is closed under extensions. Then it follows that $Z$ is in $\mod\CC$ by the middle row. Hence $\mod\CC$ is closed under cokernels of monomorphisms.

Next we show that $\mod\CC$ is closed under summands. Let $Z \in \Mod\CC$ be a summand of $Y \in \mod\CC$. We claim that there exists an exact sequence $0 \to \Omega Z \to P \to Z \to 0$ such that $P$ is in $\proj\CC$ and $\Omega Z$ is a summand of some object in $\mod\CC$, which inductively implies that $\mod\CC$ is closed under summands.
We have a split exact sequence $0\to X\to Y \to Z\to 0$ with $Y$ in $\mod\CC$, so we can use the diagram (\ref{dekai}) again and obtain the exact sequence $0 \to \Omega Z \to P \to Z \to 0$. By taking the direct sum of $0 \to Z = Z$ and the first column of (\ref{dekai}), we obtain an exact sequence $0 \to \Omega Y \to \Omega Z \oplus Z \to X \oplus Z \to 0$. Since $Y \iso X \oplus Z$ and $\Omega Y$ are in $\mod\CC$ and $\mod\CC$ is closed under extensions, it follows that $\Omega Z \oplus Z$ is in $\mod\CC$. Thus $\Omega Z$ is a summand of some object in $\mod\CC$.

By the same argument as in the proof of Lemma \ref{resol}, 
one can show that $\mod\CC$ is closed under kernels of epimorphisms, which completes the proof that $\mod\CC$ is thick in $\Mod\CC$.
The last statement of the proposition is obvious by the construction of $\mod\CC$.
\end{proof}

As a corollary, we obtain the following relation between $\mod\CC$ and the category of finitely presented $\CC$-modules (cf. \cite[Proposition 2.1]{aus66}). Recall that $\CC$ is said to have \emph{weak kernels} if for any morphism $f:C_1\to C_0$ in $\CC$, there exists a morphism $g:C_2 \to C_1$ such that $\CC(-,C_2) \xrightarrow{\CC(-,g)}  \CC(-,C_1) \xrightarrow{\CC(-,f)} \CC(-,C_0)$ is exact in $\Mod\CC$.
\begin{proposition}\label{wk}
Let $\CC$ be a skeletally small additive category. Then the category $\mod \CC$ coincides with the category of finitely presented $\CC$-modules if and only if $\CC$ has weak kernels. In this case $\mod\CC$ is an abelian category.
\end{proposition}
\begin{proof}
Suppose that $\CC$ has weak kernels and $X$ is a finitely presented $\CC$-module. Then we have an exact sequence $\CC(-,C_1) \to \CC(-,C_0) \to X \to 0$ for some objects $C_0$ and $C_1$ in $\CC$. By the Yoneda lemma, the morphism $\CC(-,C_1) \to \CC(-,C_0)$ is induced from a morphism $C_1 \to C_0$. Since $\CC$ has weak kernels, it follows that this morphism has a series of weak kernels $ \cdots \to C_2 \to C_1 \to C_0$, which yields an exact sequence $\cdots \to \CC(-,C_2) \to \CC(-,C_1) \to \CC(-,C_0) \to X \to 0$. Therefore $X$ is in $\mod\CC$.

Conversely, suppose that all finitely presented $\CC$-modules are in $\mod\CC$. Let $C_1 \to C_0$ be a morphism in $\CC$. Consider the exact sequence $0 \to X \to \CC(-,C_1) \to \CC(-,C_0) \to Z \to 0$ in $\Mod\CC$. Then $Z$ is in $\mod\CC$ by our assumption. Since $\mod\CC$ is thick in $\Mod\CC$ by Proposition \ref{mod}, we have $X \in \mod\CC$. In particular we have an exact sequence $\CC(-,C_2) \to \CC(-,C_1) \to \CC(-,C_0)$ since $X$ is finitely generated, which gives a weak kernel of $C_1 \to C_0$.

Finally we show that $\mod\CC$ is abelian if these conditions are satisfied. It is easy to see that the category of finitely presented $\CC$-modules is closed under cokernels in $\Mod\CC$, thus we only have to show that $\mod\CC$ is closed under kernels. However this is clear since $\mod\CC$ is thick in $\Mod\CC$ by Proposition \ref{mod}.
\end{proof}

Now we can prove that $\EE$ can be embedded into $\mod\PP$ as a (pre)resolving subcategory.
\begin{proposition}\label{preresol}
Let $\EE$ be a skeletally small exact category with enough projectives $\PP$ and $\P: \EE \to \mod\PP$ the Morita embedding \emph{(\ref{moritaembedding})}. Then $\P (\EE)$ is a preresolving subcategory of $\mod\PP$ and $\EE$ is exact equivalent to $\P(\EE)$. It is resolving in $\mod\PP$ if and only if $\EE$ is idempotent complete.
\end{proposition}
\begin{proof}
Since $\EE$ has enough projectives and $\P$ is an exact functor by Proposition \ref{ff}, one can easily check that $\P(\EE)$ is contained in $\mod\PP$. To show that $\P(\EE)$ is preresolving in $\mod\PP$, we check the conditions of Definition \ref{defpreresol}.

(a): Clearly $\add\P(\EE)$ contains all projective objects in $\mod\PP$, since projective objects in $\mod\PP$ are direct summands of representable functors of the form $\PP (-,P)=\P P$ for $P\in\PP$.

(b): We show that $\P(\EE)$ is closed under extensions. Suppose that $0\to \P L \to X \to \P M \to 0$ is an exact sequence in $\mod\PP$. Then there exists a conflation $\Omega M \infl P \defl M$ in $\EE$, and sending this by $\P$ gives us the following commutative diagram
\begin{eqnarray}\label{a}
\xymatrix{
0 \ar[r] & \P \Omega M \ar[r]\ar[d] & \P P \ar[r]\ar[d] & \P M \ar[r] \ar@{=}[d] & 0\\
0 \ar[r] & \P L \ar[r] & X \ar[r] & \P M \ar[r] & 0
}
\end{eqnarray}
with exact rows. The vertical morphisms exist by the projectivity of $\P P$. Then we obtain the corresponding morphism $\Omega M \to L$ in $\EE$ because $\P$ is fully faithful. By taking pushout, we get the following commutative diagram
\[
\xymatrix{
\Omega M \inflr\ar[d] & P \deflr\ar[d] & M \ar@{=}[d]\\
L \inflr & N \deflr & M 
}
\]
in $\EE$. Because $\P$ is exact, $\P$ sends this diagrams to the diagram isomorphic to (\ref{a}). This proves that $X \iso \P N$.

(c): Let $M$ be any object in $\EE$. Since there exists a conflation $\Omega M \infl P \defl M$ with $P$ in $\PP$, we have an exact sequence $0 \to \P\Omega M \to \P P \to \P M \to 0$ in $\mod\PP$. Since $\P P$ is projective, the condition (c) holds.

Next we show that $\EE$ is exact equivalent to $\P(\EE)$. We have to confirm that for morphisms $L\to M\to N$ in $\EE$, if $0\to \P L \to \P M \to \P N \to 0$ is exact, then $L \to M \to N$ is a conflation. First note that $L \to M \to N$ is a kernel and cokernel pair because $\P$ is fully faithful and $0\to \P L \to \P M \to \P N \to 0$ is exact. Take a deflation $P \defl N$ in $\EE$ with $P$ projective. Since $\P P$ is projective, $\P P \to \P N$ factors through $\P M \defl \P N$. Because $\P$ is fully faithful, this implies that $P\defl N$ also factors through $M \to N$. This shows that the composition $P \to M \to N$ is a deflation in $\EE$, and $M \to N$ has a kernel $L \to M$. By \cite[Proposition A.1.(c)]{keller}, we have that $M \to N$ is a deflation, which shows that $L \to M \to N$ is a conflation in $\EE$.

Finally we check the last statement. If $\EE$ is idempotent complete, then clearly so is $\P(\EE)$. This shows that $\P(\EE)$ is closed under summands in $\mod\PP$. Thus $\EE$ is resolving by Lemma \ref{resol}. Conversely, suppose that $\P(\EE)$ is resolving. Since $\mod \PP$ is idempotent complete by Proposition \ref{mod}, the assumption that $\P(\EE) \subset \mod\PP$ is closed under summands clearly implies that $\P(\EE)$ is idempotent complete, so is $\EE$.
\end{proof}

This proposition gives a simple criterion for a morphism in an exact category with enough projectives to be a deflation.
\begin{corollary}
Let $\EE$ be a skeletally small idempotent complete exact category with enough projectives, and let $f:X \to Y$ be a morphism in $\EE$. Then $f$ is a deflation in $\EE$ if and only if for any projective object $P \in \EE$, the induced map $\EE(P,f): \EE(P,X) \to \EE(P,Y)$ is surjective.
\end{corollary}
\begin{proof}
The ``only if'' part is clear. Suppose that $\EE(P,X)\to\EE(P,Y)$ is surjective for any projective object $P$. This is clearly equivalent to that $\P X \to \P Y$ is surjective. Since both $\P(\EE) \subset \mod\PP$ and $\mod\PP \subset \Mod\PP$ are subcategories closed under kernels of epimorphisms, so is $\P(\EE) \subset \Mod\PP$. Thus we obtain an exact sequence $0\to \P Z \to \P X \to \P Y \to 0$ in $\P(\EE)$. Since $\P$ reflects exactness and is fully faithful by Proposition \ref{preresol}, we have a conflation $Z\infl X\defl Y$, which shows that $f$ is a deflation.
\end{proof}

As a consequence, we give a correspondence between (pre)resolving subcategories of module categories and exact categories with enough projectives. To state it accurately, we need some preparation.

Let us recall the following classical Morita theorem.
\begin{definition-proposition}\label{morita}
Let $\CC$ and $\DD$ be skeletally small additive categories. The following are equivalent.
\begin{enumerate}
\item There exists an equivalence $\Mod\CC \equi \Mod\DD$.
\item There exists an exact equivalence $\mod\CC \equi \mod\DD$.
\item There exists an equivalence $\proj\CC \equi \proj\DD$.
\end{enumerate}
In this case, we say that $\CC$ and $\DD$ are \emph{Morita equivalent}. For example, $\CC$ and $\proj\CC$ are always Morita equivalent. Moreover, if $\CC$ and $\DD$ are both idempotent complete, then these are Morita equivalent if and only if $\CC \equi \DD$.
\end{definition-proposition}
\begin{proof}
The last assertion and the equivalence of (1) and (3) are classical, and we refer the reader to \cite[Proposition 2.6]{A-almost1} for the proof.

(1) $\Rightarrow$ (2):
The equivalence $\Mod\CC \equi \Mod\DD$ induces an equivalence $\proj\CC \equi \proj\DD$ since \emph{finitely generated projective} is a categorical notion. It follows from the definition of $\mod\CC$ that it also induces an exact equivalence $\mod\CC \equi \mod\DD$.

(2) $\Rightarrow$ (3):
Projective objects in the exact category $\mod\CC$ are precisely objects in $\proj\CC$ by Proposition \ref{mod}. Then the assertion follows immediately.
\end{proof}

Note that $\CC$ is idempotent complete if and only if $\proj \CC \equi \CC$. Also observe that any additive functor $F:\CC \to \DD$ induces an additive functor $\proj F:\proj\CC \to \proj\DD$.

Our aim is to classify exact categories with enough projectives such that the subcategories of projective objects are (Morita) equivalent to a fixed category $\CC$. To this purpose, we have to introduce the appropriate notion of equivalence between these categories.

\begin{definition}\label{cequiv}
Suppose that $\EE$ and $\EE'$ are skeletally small exact categories with enough projectives. Assume that $\PP(\EE)$ and $\PP(\EE')$ are Morita equivalent to $\CC$ via equivalences $F: \proj\CC \equi \proj\PP(\EE)$ and $F': \proj\CC \equi \proj\PP(\EE')$.
We say that pairs $(\EE,F)$ and $(\EE',F')$ are \emph{$\CC$-equivalent} if there exists an exact equivalence $G:\EE \equi \EE'$ such that the following diagram commutes up to natural isomorphism
\[
\xymatrix{
\proj \PP(\EE) \ar[d]_{\proj \iota} & \proj \CC \ar[l]_(.4){F}^(.4){\equi} \ar[r]^(.4){F'}_(.4){\equi} & \proj \PP(\EE') \ar[d]^{\proj \iota '}  \\
\proj\EE \ar[rr]_{\proj G}^{\equi} & &  \proj\EE' \\
}
\]
where $\iota$ and $\iota'$ are the inclusions.
\end{definition}
Note that if $\CC$, $\EE$ and $\EE'$ are idempotent complete, then the above diagram can be replaced by the following one.
\[
\xymatrix{
\PP(\EE) \ar[d]_{\iota} & \CC \ar[l]_(.35){F}^(.35){\equi} \ar[r]^(.35){F'}_(.35){\equi} & \PP(\EE') \ar[d]^{\iota '}  \\
\EE \ar[rr]_{G}^{\equi} & &  \EE' \\
}
\]

Now we are in position to state the following Morita-type theorem. This theorem says that exact categories with enough projectives with a fixed category $\CC$ of projective objects are completely classified by resolving subcategories of  $\mod\CC$.
\begin{theorem}\label{bijection}
Let $\CC$ be a skeletally small additive category.
\begin{enumerate}
\item
 There exists a bijection between the following two classes.
 \begin{enumerate}
  \item $\CC$-equivalence classes of pairs $(\EE,F)$ where $\EE$ is a skeletally small exact category with enough projectives $\PP$ such that $\PP$ is Morita equivalent to $\CC$.
  \item Preresolving subcategories of $\mod\CC$.
 \end{enumerate}
 It sends $(\EE,F)$ in {\upshape (a)} to the image of $\EE \to \mod \PP \equi \mod\CC$, and the inverse map sends $\EE$ in {\upshape (b)} to $(\EE,\mathrm{id})$ for the identity functor $\mathrm{id}:\proj\CC = \proj\CC$.
\item If $\CC$ is idempotent complete, the bijection of {\upshape (1)} restricts to a bijection between the following.
 \begin{enumerate}
  \item $\CC$-equivalence classes of pairs $(\EE,F)$ where $\EE$ is a skeletally small idempotent complete exact categories with enough projectives $\PP$ such that $\PP$ is equivalent to $\CC$.
  \item Resolving subcategories of $\mod\CC$.
 \end{enumerate}
\end{enumerate}
\end{theorem}
To prove this, we need the following preparation.
\begin{lemma}\label{eqdef}
Pairs $(\EE,F)$ and $(\EE',F')$ are $\CC$-equivalent if and only if the following diagram commutes up to natural isomorphism
\begin{equation}\label{areare}
\xymatrix{
\mod \PP(\EE) \ar[r]^{(-)\circ F}_{\equi} & \mod \CC & \mod \PP(\EE')  \ar[l]_{(-)\circ F'}^{\equi} \\
\EE \ar[u]^(.45){\P}\ar[rr]_{G}^{\equi} & &  \EE' \ar[u]_(.45){\P}\\
}
\end{equation}
where we identify $\mod\PP$ (resp. $\mod\PP'$) with $\mod (\proj \PP)$ (resp. $\mod(\proj \PP')$).
\end{lemma}
\begin{proof}
The assertion follows immediately from the following general fact.
Let $\AA$ and $\BB$ be additive categories and $K, L:\AA \rightrightarrows \BB$ fully faithful functors. Consider the composition $K', L':\BB \to \Mod\BB \rightrightarrows \Mod\AA$, where $\BB \to \Mod \BB$ is the Yoneda embedding and $\Mod\BB \rightrightarrows \Mod\AA$ are $(-)\circ K $ and $ (-)\circ L$. Then $K$ and $L$ are isomorphic if and only if $K'$ and $L'$ are isomorphic.
The details are left to the reader.
\end{proof}

\begin{proof}[Proof of Theorem \ref{bijection}]
(1):
The map from (a) to (b) is well-defined by Proposition \ref{preresol} and Lemma \ref{eqdef}.

To see that the map from (b) to (a) is well-defined, it suffices to show that the subcategory $\PP$ of projective objects in $\EE$ is Morita equivalent to $\CC$. Since $\EE$ is a preresolving subcategory of $\mod\CC$, it easily follows that $\PP$ is Morita equivalent to $\PP(\mod\CC) = \proj\CC$, which is Morita equivalent to $\CC$.

These maps are easily seen to be inverse to each other.

(2):
Suppose that $\CC$ is idempotent complete. If $\EE$ is idempotent complete exact category with enough projectives, then $\PP(\EE)$ is also idempotent complete. By Definition-Proposition \ref{morita}, one can see that the bijection in (1) restricts to the one in (2).
\end{proof}

Restring this theorem to the case of rings, one can obtain the following result, whose details are left to the reader. When $\EE$ has enough projectives $\PP$ and $\PP = \add P$ for an object $P$ in $\EE$, we call $P$ a \emph{projective generator}. 
\begin{corollary}
Let $\Lambda$ be a ring.
\begin{enumerate}
\item
 There exists a bijection between the following two classes.
 \begin{enumerate}
  \item $(\proj\Lambda)$-equivalence classes of pairs $(\EE,F)$ where $\EE$ is a skeletally small exact category with a projective generator $P$ such that $\End_\EE(P)$ is Morita equivalent to $\Lambda$.
  \item Preresolving subcategories of $\mod\Lambda$.
 \end{enumerate}
\item
 The bijection of {\upshape (1)} restricts to a bijection between the following.
 \begin{enumerate}
  \item $(\proj\Lambda)$-equivalence classes of pairs $(\EE,F)$ where $\EE$ is a skeletally small idempotent complete exact category with a projective generator $P$ such that $\End_\EE(P)$ is isomorphic to $\Lambda$.
  \item Resolving subcategories of $\mod\Lambda$.
 \end{enumerate}
\end{enumerate}
\end{corollary}

\subsection{A characterization of $\mod\PP$}
In general $\mod\PP$ has a lot of resolving subcategories. In this subsection, we shall characterize when $\P(\EE)$ and $\mod\PP$ coincide. As an application, we will give a criterion for a given exact category to be equivalent to $\mod\CC$ for some additive category $\CC$.

Let us introduce some terminologies. We say that a complex $A \xrightarrow{f} B \xrightarrow{g} C$ in an additive category $\CC$ is \emph{$\CC(\DD,-)$-exact} for a subcategory $\DD$ of $\CC$ if $\CC(D,A) \xrightarrow{f \circ (-)} \CC(D,B) \xrightarrow{g \circ (-)} \CC(D,C)$ is exact for all $D\in\DD$. Dually we define the \emph{$\CC(-,\DD)$-exactness} in the obvious way. We say that a complex $X_n \to \cdots \to X_1 \to X_0$ in an exact category $\EE$ \emph{decomposes into conflations} if there exist a commutative diagram in $\EE$
\[
\xymatrix@C=1.3em@R=.3em{
A_{n+1} \ar@{ >->}[rd] & & & & A_{2}\ar@{ >->}[rd] & & & & A_0\\
& X_n \ar@{->>}[rd] \ar[rr] & & \cdots \ar@{->>}[ru] \ar[rr] & & X_1 \ar[rr] \ar@{->>}[rd] & & X_0 \ar@{->>}[ru]& \\
& & A_{n} \ar@{ >->}[ru] & & & & A_1 \ar@{ >->}[ru] & &
}
\]
where $A_{n+1}\infl X_n \defl A_n$, $\cdots$, $A_1 \infl X_0 \defl A_0$ are conflations. 
\begin{theorem}\label{rec1}
Let $\EE$ be a skeletally small exact category.
\begin{enumerate}
\item
 The following conditions are equivalent.
 \begin{enumerate}
  \item There exists a skeletally small additive category $\CC$ such that $\EE$ is exact equivalent to $\mod\CC$.
  \item $\EE$ has enough projectives $\PP$, idempotent complete, and for any $\EE(\PP,-)$-exact complex 
\[
\cdots \xrightarrow{f_3} P_2 \xrightarrow{f_2} P_1 \xrightarrow{f_1} P_0
\]
with all terms in $\PP$, the morphism $f_1:P_1 \to P_0$ can be factored as a deflation followed by an inflation.
 \end{enumerate}
\item
 The following conditions are equivalent.
 \begin{enumerate}[label={\upshape(\alph*$'$)}]
  \item There exists a skeletally small additive category $\CC$ with weak kernels such that $\EE$ is exact equivalent to $\mod\CC$.
  \item $\EE$ has enough projectives, idempotent complete, and any morphism in $\PP$ can be factored as a deflation followed by an inflation.
 \end{enumerate}
\end{enumerate}
\end{theorem}
\begin{proof}
(a) $\Rightarrow$ (b): We may assume that $\EE=\mod\CC$. Then $\mod\CC$ is idempotent complete and has enough projectives by Proposition \ref{mod}. Note that in $\mod\CC$, the notion of $(\mod\CC)(\proj\CC,-)$-exactness is equivalent to exactness in $\Mod\CC$ by the Yoneda Lemma. Suppose that there exists an exact sequence
\[
\cdots \xrightarrow{f_3} P_2 \xrightarrow{f_2} P_1 \xrightarrow{f_1} P_0
\]
in $\Mod\CC$ whose all terms are in $\proj\CC$. Then the cokernel of $f_1$ is clearly in $\mod\CC$. Since $\mod\CC\subset\Mod\CC$ is resolving and $\coker f_1$ is in $\mod\CC$, we must have $\im f_i$ is in $\mod\CC$ for all $i \geq 1$. This clearly implies that the above complex decomposes into conflations in $\mod\CC$. Especially, $f_1$ is decomposed into $P_1 \defl \im f_1 \infl P_0$.

(a)$' \Rightarrow$ (b)$'$:
Let $f:P_1 \to P_0$ be a morphism in $\proj\CC$. Then $\coker f$ is finitely presented, hence is in $\mod\CC$ by Proposition \ref{wk}. Since $\mod\CC$ is resolving in $\Mod\CC$, the same argument as above implies that $P_1 \defl \im f \infl P_0$ is the desired factorization in $\mod\CC$.

(b) $\Rightarrow$ (a): Recall that the Morita embedding $\P$ realizes $\EE$ as a resolving subcategory of $\mod\PP$ by Proposition \ref{preresol}. 
We must check that $\P :\EE \to \mod\PP$ is essentially surjective. Let $X$ be an arbitrary object in $\mod\PP$. By the definition of $\mod\PP$ we have an exact sequence $\cdots \to \P P_2 \to \P P_1 \to \P P_0 \to X \to 0$ with $P_i$ in $\PP$ for all $i \geq 0$. By (b), the corresponding $P_1 \to P_0$ in $\PP$ factorizes into a deflation followed by an inflation in $\EE$. Then it immediately follows that the cokernel of $\P P_1 \to \P P_0$ is in $\P(\EE)$, which proves that $\P$ is essentially surjective.

(b)$' \Rightarrow$ (a)$'$: We only have to check that $\PP$ has weak kernels. By (b)$'$, for any $P_1 \to P_0$ in $\PP$ there exist conflations $X_1 \infl P_1 \defl X$ and $X \infl P_0 \defl X'$. Moreover since $\EE$ has enough projectives, there exists a conflation $X_2 \infl P_2 \defl X_1$. It is clear that the composition $P_2 \defl X_1 \infl P_1$ is a weak kernel of $P_1 \to P_0$.
\end{proof}

\section{Exact categories with enough projectives and injectives}
In this section, we study exact categories with both enough projectives and enough injectives. 
We will show that an arbitrary exact category with enough projectives and injectives can be realized as a preresolving-precoresolving subcategory of the exact category $\XXX_\WW$, the exact category associated to a Wakamatsu tilting subcategory $\WW$ of a module category. 
We freely use results in the previous section.
\subsection{Wakamatsu tilting subcategories}
Wakamatsu introduced a generalization of the classical concept of tilting modules, called \emph{Wakamatsu tilting modules}, which have possibly infinite projective dimensions \cite{wa,mr}. They are also called \emph{semi-dualizing modules} by some authors \cite{ch,aty}.
With a Wakamatsu tilting module $W$, we can associate an exact category $\XXX_W$ with enough projectives and injectives. Typical examples are the category of Gorenstein projective $\Lambda$-modules $\GP \Lambda$ and the $\Ext$-orthogonal category $^\perp U$ for a cotilting module $U$.
In what follows, we introduce a categorical analogue of Wakamatsu tilting modules, called \emph{Wakamatsu tilting subcategories}.

In this subsection, \emph{we denote by $\EE$ an exact category with enough projectives}. Let $^\perp \WW$ denote the category consisting of objects $X$ in $\EE$ satisfying $\Ext^{>0}_\EE(X,\WW)=0$.

For a subcategory $\WW$ of $\EE$, we first introduce the following subcategory $\XXX_\WW$ of $\EE$ in which objects in $\WW$ behaves like an injective cogenerator, see Proposition \ref{xxww}. We say that a subcategory $\WW$ of an exact category $\EE$ is \emph{self-orthogonal} if $\Ext^{>0}_\EE(\WW,\WW)=0$.
\begin{definition}\label{xxxww}
Let $\EE$ be an exact category with enough projectives and $\WW$ an additive subcategory of $\EE$.
\begin{enumerate}
\item We denote by $\XXX_\WW$ the full subcategory of $\EE$ consisting of all objects $X^0\in {}^\perp \WW$ such that there exist conflations $X^0 \infl W^0 \defl X^1$, $X^1 \infl W^1 \defl X^2$, $\cdots$ with $W^i \in\WW$ and $X^i \in {} ^\perp\WW$ for $i \geq 0$.

\item We say that $\WW$ is a \emph{Wakamatsu tilting subcategory} if it satisfies the following conditions.
\begin{enumerate}
\item $\WW$ is self-orthogonal.
\item $\XXX_\WW$ contains all projective objects in $\EE$.
\end{enumerate}
\end{enumerate}
\end{definition}

Note that the subcategory $\PP(\EE)$ of $\EE$ consisting of all projective objects is always Wakamatsu tilting in $\EE$.
A $\Lambda$-module $W$ for a ring $\Lambda$ is said to be a \emph{Wakamatsu tilting module} or a \emph{semi-dualizing module} if $\add W$ is a Wakamatsu tilting subcategory of $\mod\Lambda$. Our definition coincides with the usual one.

Basic properties of the category $\XXX_\WW$ are as follows.

\begin{proposition}\label{xxww}
Let $\EE$ be an exact category with enough projectives and $\WW$ an additive self-orthogonal subcategory of $\EE$. 
\begin{enumerate}
\item $\XXX_\WW$ is closed under extensions, kernels of deflations and summands in $\EE$ and it is thick in $^\perp\WW$.
\end{enumerate}
If moreover $\WW$ is Wakamatsu tilting, then the following hold.
\begin{enumerate}[resume]
\item $\XXX_\WW$ is a resolving subcategory of $\EE$. 
\item $\XXX_\WW$ has enough projectives and injectives.
\item $P$ in $\XXX_\WW$ is projective in $\XXX_\WW$ if and only if $P$ is projective in $\EE$.
\item $I$ in $\XXX_\WW$ is injective in $\XXX_\WW$ if and only if $I$ is in $\add \WW$.
\end{enumerate}
\end{proposition}
\begin{proof}
The same proof of Proposition 5.1 in \cite{applications} applies. For the convenience of the reader, we shall give a proof. 

(1):
We first show that $\XXX_\WW$ is closed under extensions. Suppose that $A\infl B\defl C$ is a conflation in $\EE$ with A and C in $\XXX_\WW$. It is easy to see that $^\perp\WW$ is closed under extensions, thus we have $B$ is in $^\perp\WW$. By the definition of $\XXX_\WW$, there exist conflations $A \infl W \defl A^1$ and $C \infl W' \defl C^1$ such that $W$ and $W'$ are in $\WW$ and $A^1$ and $C^1$ are in $\XXX_\WW$, 

Consider the pushout diagram
\[
\xymatrix{
A \inflr\infld & B \infld\deflr & C \ar@{=}[d] \\
W \inflr\defld & U \defld\deflr & C\\
A^1 \ar@{=}[r] & A^1 &
}
\]
in $\EE$. Since $C$ is in $\XXX_\WW \subset {}^\perp\WW$, the middle row splits. Thus we may assume that $U = W \oplus C$. We then have the commutative diagram.
\[
\xymatrix{
B \inflr \ar@{=}[d] & W \oplus C \deflr \infld& A^1 \infld \\
B \inflr & W \oplus W' \deflr \defld & D \defld \\
 & C^1 \ar@{=}[r] & C^1
}
\]
where the middle columns is a direct sum of $W = W \to 0$ and $C \infl W' \defl C^1$. Since $A^1$ and $C^1$ are in $\XXX_\WW$, if follows that $D$ is an extension of objects in $\XXX_\WW$. By considering the middle row, one may proceed this process to see that $B$ is in $\XXX_\WW$.

We can prove that $\XXX_\WW$ is closed under kernels of deflations and summands by the same argument as the proof in Proposition \ref{mod}, so we leave it to the reader.

Next we show that $\XXX_\WW$ is a thick subcategory of $^\perp\WW$. Obviously it suffices to see that $\XXX_\WW$ is closed under cokernels of inflations in $^\perp\WW$. Let $A \infl B\defl C$ be a conflation with all terms in $^\perp\WW$ and assume that $A$ and $B$ are in $\XXX_\WW$. We have a conflation $A \infl W \defl A'$ with $W$ in $\WW$ and $A'$ in $\XXX_\WW$. Then consider the following pushout diagram.
\[
\xymatrix{
A \inflr \infld & B \infld \deflr & C \ar@{=}[d]\\
W \inflr \defld & D \defld \deflr & C \\
A' \ar@{=}[r] & A' &
}
\]
Because $A'$ and $B$ are in $\XXX_\WW$ and $\XXX_\WW$ is closed under extensions, $D$ must be in $\XXX_\WW$. On the other hand, the middle row splits because $C$ is in $^\perp\WW$, which implies that $C$ is a summand of $D$. Since $\XXX_\WW$ is closed under summands, it follows that $C$ is in $\XXX_\WW$, which shows that $\XXX_\WW \subset {}^\perp\WW$ is thick.

(2)-(5):
If $\WW$ is in addition Wakamatsu tilting, then by definition all projective objects are in $\XXX_\WW$, which clearly implies that $\XXX_\WW$ is resolving in $\EE$.
The remaining assertions are obvious from the definition.
\end{proof}

\subsection{Morita-type theorem}
We assume that \emph{$\EE$ is a skeletally small exact category with enough projectives $\PP=\PP(\EE)$ and enough injectives $\II=\II(\EE)$.} 
The following result gives an explicit description of the image of the Morita embedding functor $\P:\EE \to \mod\PP$ (\ref{moritaembedding}) in terms of a Wakamatsu tilting subcategory.

\begin{theorem}\label{pe}
Let $\EE$ be a skeletally small exact category with enough projectives $\PP$ and enough injectives $\II$. For the Morita embedding $\P:\EE \to \mod\PP$, we set $\WW = \P(\II)$. Then $\WW$ is a Wakamatsu tilting subcategory of $\mod\PP$, and $\P(\EE)$ is a preresolving-precoresolving subcategory of $\XXX_\WW$ (Definition \ref{defpreresol}). Moreover it is resolving-coresolving if and only if $\EE$ is idempotent complete.
\end{theorem}
\begin{proof}
First we see that $\P(\EE)$ is contained in $\XXX_\WW$. Since $\EE$ has enough injectives $\II$, for any object $X\in\EE$, there exist conflations $X \infl I^0 \defl X^1$, $X^1 \infl I^1 \defl X^2$, $\cdots$ with $I^i$ injective for all $i$. Applying $\P$, we obtain conflations $\P X \infl \P I^0 \defl \P X^1$, $\P X^1 \infl \P I^1 \defl \P X^2$, $\cdots$ with $\P I^i$ in $\WW$ for all $i$. Recall that $\P$ preserves all extension groups by Proposition \ref{ff}, thus $\P(\EE) \subset {}^\perp\WW$ holds. Therefore, the conflations show that $\P X$ is in $\XXX_\WW$.

Next we show that $\WW$ is a Wakamatsu tilting subcategory of $\mod\PP$. Since $\P(\EE) \subset \XXX_\WW \subset {}^\perp\WW$ holds, $\WW$ is clearly self-orthogonal and $\P(\PP)$ is contained in $\XXX_\WW$. On the other hand, $\XXX_\WW$ is closed under direct sums and summands by Proposition \ref{xxww}. Thus it follows that $\proj\PP=\add \P(\PP) \subset \XXX_\WW$, which implies that $\WW$ is Wakamatsu tilting.

It remains to prove that $\P(\EE)$ is a preresolving-precoresolving subcategory of $\XXX_\WW$. 
Since $\P(\EE) \subset \mod\PP$ and $\XXX_\WW \subset \mod\PP$ are extension-closed subcategories, it follows that $\P(\EE)$ is closed under extensions in $\XXX_\WW$.
By Proposition \ref{xxww}(4)(5), projective (resp. injective) objects in $\XXX_\WW$ are precisely objects in $\proj\PP$ (resp. $\add\WW$). Thus $\add \P(\EE)$ contains both projective and injective objects in $\XXX_\WW$. 
Moreover the images of projective and injective resolutions in $\EE$ under $\P:\EE \to \XXX_\WW$ yield desired conflations in Definition \ref{defpreresol}(c). Thus $\P(\EE)$ is preresolving-precoresolving in $\XXX_\WW$.

If $\EE$ is idempotent complete, then $\P(\EE)$ is closed under summands in $\mod\PP$, which implies that $\P(\EE)$ is resolving and coresolving in $\XXX_\WW$ by Lemma \ref{resol}. Conversely, suppose that $\P(\EE)$ is resolving or coresolving in $\XXX_\WW$. Since $\mod\PP$ is idempotent complete and $\XXX_\WW$ is closed under summands, $\P(\EE)$ is clearly idempotent complete, hence so is $\EE$. 
\end{proof}
Note that in the case of Frobenius categories, Theorem \ref{pe} was shown in \cite[Theorem 4.2]{chen} 

Conversely, any preresolving-precoresolving subcategories of $\XXX_\WW$ clearly have enough projectives and injectives. Hence one obtains the following classification of exact categories with enough projectives and injectives. First we modify the notion of the $\CC$-equivalence defined in Definition \ref{cequiv}.
\begin{definition}\label{cwequiv}
Suppose that $\EE$ and $\EE'$ are skeletally small exact categories with enough projectives and injectives. Assume that $\PP(\EE)$ and $\PP(\EE')$ are Morita equivalent to $\CC$ via equivalences $F: \proj\CC \equi \proj\PP(\EE)$ and $F': \proj\CC \equi \proj\PP(\EE')$, and that $\II(\EE)$ and $\II(\EE')$ are Morita equivalent to $\WW$ via equivalences $H: \proj\WW \equi \proj\II(\EE)$ and $H': \proj\WW \equi \proj\II(\EE')$
We say that pairs $(\EE,F,H)$ and $(\EE',F',H')$ are \emph{$(\CC,\WW)$-equivalent} if there exists an exact equivalence $G:\EE \equi \EE'$ which makes the following diagrams commute up to natural isomorphism
\begin{align*}
\xymatrix{
\proj \PP(\EE) \ar[d]_{\proj \iota} & \proj \CC \ar[l]_(.4){F}^(.4){\equi} \ar[r]^(.4){F'}_(.4){\equi} & \proj \PP(\EE') \ar[d]^{\proj \iota '}  \\
\proj\EE \ar[rr]_{\proj G}^{\equi} & &  \proj\EE' \\
}
&&
\xymatrix{
\proj \II(\EE) \ar[d]_{\proj \iota} & \proj \WW \ar[l]_(.45){H}^(.45){\equi} \ar[r]^(.45){H'}_(.45){\equi} & \proj \II(\EE') \ar[d]^{\proj \iota '}  \\
\proj\EE \ar[rr]_{\proj G}^{\equi} & &  \proj\EE' \\
}
\end{align*}

where $\iota$ and $\iota'$ are the inclusions.
\end{definition}

\begin{theorem}\label{wakamatsucorresp}
Let $\CC$ be an additive category and $\WW$ a Wakamatsu tilting subcategory in $\mod\CC$.
\begin{enumerate}
 \item There exists a bijection between the following two classes.
 \begin{enumerate}
  \item $(\CC,\WW)$-equivalence classes of pairs $(\EE,F,H)$ where $\EE$ is a skeletally small exact category with enough projectives $\PP$ and enough injectives $\II$ such that $\PP$ is Morita equivalent to $\CC$ and $\II$ is Morita equivalent to $\WW$.
  \item Preresolving-precoresolving subcategories of $\XXX_\WW$.
 \end{enumerate}
 Moreover, any exact categories with enough projectives and injectives occur in this way. 
\item Suppose that $\CC$ and $\WW$ are idempotent complete. Then the bijection of {\upshape (1)} restricts to a bijection between the following.
 \begin{enumerate}
  \item $(\CC,\WW)$-equivalence classes of pairs $(\EE,F,H)$ where $\EE$ is a skeletally small idempotent complete exact category with enough projectives $\PP$ and enough injectives $\II$ such that $\PP$ is equivalent to $\CC$ and $\II$ is equivalent to $\WW$.
  \item Resolving-coresolving subcategories of $\XXX_\WW$
 \end{enumerate}
\end{enumerate}
\end{theorem}
\begin{proof}
The proof of Theorem \ref{bijection} applies.
\end{proof}

One can conclude the following results on a ring immediately.
\begin{corollary}
Let $\Lambda$ be a ring and $W$ a Wakamatsu tilting $\Lambda$-module.
\begin{enumerate}
 \item There exists a bijection between the following two classes.
 \begin{enumerate}
  \item $(\proj\Lambda,\add W)$-equivalence classes of pairs $(\EE,F,H)$ where $\EE$ is a skeletally small exact category with a projective generator $P$ and an injective cogenerator $I$ such that $\End_\EE(P)$ is Morita equivalent to $\Lambda$ and $\End_\EE(I)$ is Morita equivalent to $\End_\Lambda(W)$.
  \item Preresolving-precoresolving subcategories of $\XXX_W$.
 \end{enumerate}
 Moreover, any exact categories with projective generators and injective cogenerators occur in this way.
\item The bijection of {\upshape (1)} restricts to a bijection between the following.
 \begin{enumerate}
  \item $(\proj\Lambda,\add W)$-equivalence classes of pairs $(\EE,F,H)$ where $\EE$ is a skeletally small idempotent complete exact category with a projective generator $P$ and an injective cogenerator $I$ such that $\End_\EE(P)$  is isomorphic to $\Lambda$ and $\End_\EE(I)$ is isomorphic to $\End_\Lambda(W)$.
  \item Resolving-coresolving subcategories of $\XXX_W$.
 \end{enumerate}
\end{enumerate}
\end{corollary}

Next we consider the case of Frobenius categories. Recall that an exact category $\EE$ is called \emph{Frobenius} if it has enough projectives and injectives, and projective objects and injective objects coincide. A typical example of a Frobenius category is given by the following.
\begin{definition}\label{gpc}
For a skeletally small additive category $\CC$, we denote by $\GP \CC$ the exact subcategory $\XXX_\WW$ of $\mod\CC$ for the Wakamatsu tilting subcategory $\WW=\proj \CC$ of $\mod\CC$. Modules in $\GP \CC$ are called \emph{Gorenstein projective}.  Similarly for a ring $\Lambda$, we denote by $\GP \Lambda$ the exact subcategory $\XXX_\Lambda$ of $\mod\Lambda$.
\end{definition}
It follows from Proposition \ref{xxww} that $\GP \CC$ is a Frobenius category. We remark that different names such as Cohen-Macaulay or totally reflexive are used for $\GP \CC$. 

Using the bijection of Theorem \ref{wakamatsucorresp}, one can easily show the following results about Frobenius categories, where we use the notion of $\CC$-equivalence defined in Definition \ref{cequiv}.

\begin{theorem}
Let $\CC$ be a skeletally small additive category.
\begin{enumerate}
\item There exists a bijection between the following two classes.
 \begin{enumerate}
  \item $\CC$-equivalence classes of pairs $(\EE,F)$ where $\EE$ is a skeletally small Frobenius category such that $\PP(\EE)$ is Morita equivalent to $\CC$.
  \item Preresolving-precoresolving subcategories of $\GP \CC$.
\end{enumerate}
\item Suppose that $\CC$ is idempotent complete. Then the bijection of {\upshape (1)} restricts to a bijection between the following.
 \begin{enumerate}
  \item $\CC$-equivalence classes of pairs $(\EE,F)$ where $\EE$ is a skeletally small idempotent complete Frobenius category such that $\PP(\EE)$ is equivalent to $\CC$.
  \item Resolving-coresolving subcategories of $\GP \CC$.
 \end{enumerate}
\end{enumerate}
\end{theorem}

\begin{corollary}
Let $\Lambda$ be a ring.
\begin{enumerate}
\item There exists a bijection between the following two classes.
 \begin{enumerate}
  \item $(\proj\Lambda)$-equivalence classes of pairs $(\EE,F)$ where $\EE$ is a skeletally small Frobenius category $\EE$ with a projective generator  $P$ such that $\End_\EE(P)$ is Morita equivalent to $\Lambda$.
  \item Preresolving-precoresolving subcategories of $\GP \Lambda$.
 \end{enumerate}
\item
 The bijection of {\upshape (1)} restricts to a bijection between the following.
 \begin{enumerate}
  \item $(\proj\Lambda)$-equivalence classes of pairs $(\EE,F)$ where $\EE$ is a skeletally small idempotent complete Frobenius category $\EE$ with a projective generator $P$ such that $\End_\EE(P)$ is isomorphic to $\Lambda$.
  \item Resolving-coresolving subcategories of $\GP \Lambda$.
 \end{enumerate}
\end{enumerate}
\end{corollary}

\subsection{A characterization of $\XXX_\WW$}
Let $\EE$ be a skeletally small exact category with enough projectives $\PP$ and enough injectives $\II$. We have the Morita embedding $\P: \EE \to \XXX_\WW$ for $\WW = \P(\II)$ by Theorem \ref{pe}. The following theorem characterizes when $\P(\EE)$ and $\XXX_\WW$ (or $\GP\PP$) actually coincide. 

\begin{theorem}\label{rec2}
Let $\EE$ be a skeletally small exact category. Then the following are equivalent.
\begin{enumerate}
\item There exist a skeletally small additive category $\CC$ and a Wakamatsu tilting subcategory $\WW$ of $\mod\CC$ such that $\EE$ is exact equivalent to $\XXX_\WW$.
\item $\EE$ is idempotent complete and has enough projectives $\PP$ and enough injectives $\II$, and any $\EE(\PP,-)$-exact and $\EE(-,\II)$-exact complex
\[
\cdots \to P_2 \to P_1 \to P_0 \to I^0 \to I^1 \to I^2 \to \cdots
\] with $P_i$ in $\PP$ and $I^i$ in $\II$ for $i \geq 0$ decomposes into conflations.
\end{enumerate}
\end{theorem}
To prove this, we need the following technical lemma.
\begin{lemma}\label{lemmaperp}
Let $\CC$ be a skeletally small additive category and $\WW$ an additive self-orthogonal subcategory of $\mod\CC$. Suppose that $\cdots \to M_2 \to M_1 \to M_0 \to M_{-1} \to M_{-2} \to \cdots$ is a complex in $^\perp\WW$ which decomposes into conflations $X_{i+1} \infl M_i \defl X_i$ in $\mod\CC$ for all $i \in \Z$. If this complex is $(\mod\CC)(-,\WW)$-exact, then $X_i$ is in $^\perp\WW$ for all $i \in \Z$.
\begin{proof}
Put $\FF:=\mod\CC$ for simplicity. Since $M_i$'s are in $^\perp\WW$, we have an exact sequence
\[
\FF(M_i, \WW) \to \FF(X_{i+1},\WW) \to \Ext^1_\FF(X_i,\WW) \to 0.
\]
by the long exact sequence of $\Ext$. We first show that $\Ext^1_\FF(X_i,\WW)$ vanishes. Consider the following commutative diagram.
\[
\xymatrix{
0 \ar[r] & \FF(X_{i+1},\WW) \ar[r] & \FF(M_{i+1}, \WW) \ar[r] & \FF(M_{i+2},\WW) \\
 & \FF(M_i,\WW) \ar[u] \ar[r] & \FF(M_{i+1}, \WW) \ar[r]\ar@{=}[u] & \FF(M_{i+2},\WW)\ar@{=}[u]
}
\]
Since given complex is $\FF(-,\WW)$-exact, the bottom row is exact. One can easily show that $M_{i+2} \to M_{i+1} \to X_{i+1}$ is a cokernel diagram, which implies that the top row is also exact. Thus we see that $\FF(M_i, \WW) \to \FF(X_{i+1},\WW)$ is surjective, which shows that $\Ext^1_\FF(X_i,\WW)=0$. Since $\Ext^j_\FF(X_i,\WW) = \Ext^1_\FF(X_{i+j-1},\WW) =0$ for $j \geq 1$ by the dimension shift, it follows that $X_i$ is in $^\perp\WW$.
\end{proof}
\end{lemma}
\begin{proof}[Proof of Theorem \ref{rec2}]
(1) $\Rightarrow$ (2): We may assume that $\EE=\XXX_\WW$, and in this case $\II$ coincides with $\add\WW$ by Proposition \ref{xxww}(5). Suppose that
\[
\cdots \to P_2 \to P_1 \to P_0 \to I^0 \to I^1 \to I^2 \to \cdots
\]
 is a complex in $\XXX_\WW$ satisfying the condition. Since $\EE(\PP,-)$-exactness is equivalent to exactness in $\Mod\PP$, we can decompose it into short exact sequences $X_{i+1} \infl P_i \defl X_i$ and $X^i \infl I^i \defl X^{i+1}$ for $i \geq 0$ with $X_0 = X^0$. Since the $X_0$ is obviously in $\mod\CC$ and $\mod\CC \subset \Mod\CC$ is thick, all $X_i$ and $X^i$ are in $\mod\CC$ for $i \geq 0$.
 
By Lemma \ref{lemmaperp}, $X^i$ is in $^\perp\WW$ for all $i \geq 0$. Thus the conflations $X^i \infl I^i \defl X^{i+1}$, $X^{i+1} \infl I^{i+1} \defl X^{i+2}$, $\cdots$ imply that $X^i$ is in $\XXX_\WW$ for $i \geq 0$. Since $\XXX_\WW \subset \mod\CC$ is resolving, $X_i$ is also in $\XXX_\WW$ for $i \geq 0$, which proves that this complex decomposes into conflations in $\XXX_\WW$.

(2) $\Rightarrow$ (1): Suppose that (2) holds. Since the Morita embedding $\P :\EE \to \XXX_\WW$  in Theorem \ref{pe} is fully faithful for $\WW = \P(\II)$, it suffices to show that it is essentially surjective. Let $X$ be an object in $\XXX_\WW$. Since $X$ is in $\mod\PP$ and in $\XXX_\WW$, there exists a complex 
\begin{equation}\label{pi1}
\cdots \to P_2 \to P_1 \to P_0 \to I^0 \to I^1 \to I^2 \to \cdots
\end{equation}
 with $P_i$ in $\PP$ and $I^i$ in $\II$ for $i \geq 0$ satisfying the following properties: 
\begin{equation}\label{pi2}
 \cdots \to \P P_2 \to \P P_1 \to \P P_0 \to \P I^0 \to \P I^1 \to \P I^2 \to \cdots
\end{equation}
 is exact, $X$ is the cokernel of $\P P_1 \to \P P_0$ and all the kernels of $\P I^i \to \P I^{i+1}$ are in $^\perp\WW$ for $i \geq 0$. Decompose (\ref{pi2}) into short exact sequences $X_{i+1} \infl \P P_i \defl X_i$ and $X^i \infl \P I^{i} \defl X^{i+1}$ in $\Mod\PP$ for $i \geq 0$ with $X_0 = X = X^0$. Since $X$ is in $\XXX_\WW \subset \mod\PP$ and $\mod\PP$ is thick in $\Mod\PP$, it follows that $X_i$ and $X^i$ are in $\mod\PP$ for all $i \geq 0$. Moreover, $X^i$ is clearly in $\XXX_\WW$ for $i \geq 0$. We have that $X_i$ is in $\XXX_\WW$ for $i \geq 0$ since $\XXX_\WW$ is a resolving subcategory of $\mod\PP$. It follows that (\ref{pi2}) is $(\mod\PP)(\P(\PP),-)$-exact and $(\mod\PP)(-,\WW)$-exact, which shows that (\ref{pi1}) is $\EE(\PP,-)$-exact and $\EE(-,\II)$-exact. 
Hence by (2) the complex (\ref{pi1}) can be decomposed into conflations in $\EE$. Since $\P$ is an exact functor, it immediately follows that $X$ is in $\P(\EE)$.
\end{proof}

In the case of Frobenius categories, this gives an internal characterization of categories of Gorenstein projective modules $\GP \CC$.
\begin{corollary}\label{rec3}
Let $\EE$ be a skeletally small exact category. The following are equivalent.
\begin{enumerate}
\item There exists a skeletally small additive category $\CC$ such that $\EE$ is exact equivalent to $\GP \CC$.
\item $\EE$ is an idempotent complete Frobenius category, and any $\EE(\PP,-)$-exact and $\EE(-,\PP)$-exact complexes \[
\cdots \to P_2 \to P_1 \to P_0 \to P_{-1} \to P_{-2} \to \cdots
\]
in $\PP$ decomposes into conflations in $\EE$.
\end{enumerate}
\end{corollary}
Note that this class of Frobenius categories $\EE$ such that the image of the Morita embedding $\P:\EE \to \mod\PP$ coincides with $\GP \PP$ are called \emph{standard} in \cite{chen}. The above corollary gives an intrinsic characterization of standard Frobenius categories.
Also the category $\XXX_W$ for a Wakamatsu tilting $\Lambda$-module $W$ and the category $\GP \Lambda$ for a ring $\Lambda$ can be characterized by similar conditions to Theorem \ref{rec2}(2) and Corollary \ref{rec3}(2), which we leave to the reader.

\section{Exact categories with enough projectives and injectives and higher kernels}
In this section we study the special class of Wakamatsu tilting subcategories, \emph{cotilting subcategories}, and study its relationship with \emph{higher kernels}. More precisely, we show that an exact category $\EE$ with enough projectives and injectives is equivalent to $\XXX_\WW$ for a cotilting subcategory $\WW$ of a module category if and only if $\EE$ has higher kernels.
\subsection{Cotilting subcategories}
First we introduce the notion of $n$-cotilting subcategories, which is a generalization of cotilting modules, as we will see in Proposition \ref{wakamatsucotilting} below.
\begin{definition}
Let $\EE$ be an exact category with enough projectives and $\WW$ an additive subcategory of $\EE$. We say that $\WW$ is an \emph{$n$-cotilting subcategory} if it satisfies the following conditions.
\begin{enumerate}
\item For every $W \in \WW$, we have $\id W \leq n$, that is, $\Ext_\EE^{>n}(-,W)=0$.
\item $\WW$ is self-orthogonal, that is, $\Ext^{>0}_\EE(\WW,\WW)=0$.
\item The categories $^\perp\WW$ and $\XXX_\WW$ coincide (see Definition \ref{xxxww} for the category $\XXX_\WW$).
\end{enumerate}
Note that an $n$-cotilting subcategory is always Wakamatsu tilting.
\end{definition}

Next we study basic properties of cotilting subcategories.
For a subcategory $\XX$ of an exact category $\EE$, we denote by $\widehat{\XX}^n$ the subcategory of $\EE$ consisting of all objects $M$ such that there exists a complex which decomposes into conflations in $\EE$
\[
0 \to X_n \to X_{n-1} \to \cdots \to X_0 \to M \to 0,
\]
where $X_i$ is in $\XX$ for $0 \leq i \leq n$. We write $\widehat{\XX}$ for the subcategory of $\EE$ whose objects are those in $\widehat{\XX}^n$ for some $n$. We set $\widehat{\XX}^0 = \XX$ and $\widehat{\XX}^n = 0$ for $n < 0$.

\begin{proposition}\label{cotiltingcond}
Let $\EE$ be an exact category with enough projectives and $\WW$ an additive self-orthogonal subcategory of $\EE$. For an integer $n\geq 0$, the following are equivalent.
\begin{enumerate}
\item $\WW$ is an $n$-cotilting subcategory of $\EE$.
\item $\widehat{\XXX_\WW}^n = \EE$.
\item $\XXX_\WW$ contains all projective objects, and for any complex which decomposes into conflations
 \[
0 \to X_n \to X_{n-1} \to \cdots \to X_0 \to M \to 0,
\]
if $X_i$ is in $^\perp\WW$ for $0\leq i < n$, then $X_n$ is in $\XXX_\WW$.
\end{enumerate}
\end{proposition}
\begin{proof}
(1) $\Rightarrow$ (3):
All projective objects are clearly contained in $^\perp\WW$, thus in $\XXX_\WW$. Let 
 \[
0 \to X_n \to X_{n-1} \to \cdots \to X_0 \to M \to 0
\]
be a complex which decomposes into conflations with $X_i$ is in $^\perp\WW$ for $0\leq i < n$. The dimension shift argument shows that $\Ext^{>0}_\EE(X_n,\WW)=\Ext^{>n}_\EE(M,\WW)=0$ since $\id \WW \leq n$. Hence $X_n$ is in $^\perp\WW = \XXX_\WW$, which shows (3).

(3) $\Rightarrow$ (2): This is clear since $\EE$ has enough projectives and all projectives are in $\XXX_\WW$.

(2) $\Rightarrow$ (1):
For any $M$ in $\EE= \widehat{\XXX_\WW}^n$, there exist conflations
\begin{equation}\label{bunkai}
X_n \infl X_{n-1} \defl M_{n-1},\quad M_{n-1} \infl X_{n-2} \defl M_{n-2},\quad \cdots,\quad M_1 \infl X_0 \defl M
\end{equation}
 where $X_i$ is in $\XXX_\WW$ for $0 \leq i \leq n$. The dimension shift argument shows that $\Ext^{>n}_\EE(M,\WW) = \Ext^{>0}_\EE(X_n,\WW) $ $= 0$, which proves that $\id \WW \leq n$. 

It remains to check that $^\perp\WW = \XXX_\WW$. If $M$ is in $^\perp\WW$, then all terms in (\ref{bunkai}) are in $^\perp\WW$ because $^\perp\WW$ is resolving in $\EE$. On the other hand, according to Proposition \ref{xxww}, the category $\XXX_\WW$ is closed under cokernels of inflations in $^\perp\WW$. This clearly implies that $M_{n-1},\cdots,M_1,M$ are in $\XXX_\WW$, thus we have $^\perp\WW = \XXX_\WW$.
\end{proof}

Let us describe the relation between our cotilting subcategories and classical cotilting modules.
The notion of cotilting modules over artin $R$-algebras is the dual notion of \emph{tilting} modules, and both are widely studied in the representation theory of algebras. Let $\Lambda$ be an artin $R$-algebra and $U$ a finitely generated $\Lambda$-module. Then $U$ is called a  \emph{cotilting module} if it satisfies the following conditions.
\begin{enumerate}
\item $\id U_\Lambda$ is finite.
\item $U$ is self-orthogonal.
\item $D \Lambda$ belongs to $\widehat{\add U}$.
\end{enumerate}

The following statement illustrates the relation between our definition and the classical one.
\begin{proposition}\label{wakamatsucotilting}
Let $\Lambda$ be an artin $R$-algebra and $\WW $ a subcategory of $\mod\Lambda$ satisfying $\add\WW = \WW$. For an integer $n \geq 0$, the following are equivalent.
\begin{enumerate}
\item $\WW$ is an $n$-cotilting subcategory of $\mod\Lambda$.
\item There exists a cotilting $\Lambda$-module $U$ with $\id U \leq n$ such that $\WW = \add U$.
\end{enumerate}
\end{proposition}
\begin{proof}
We refer the reader to \cite[Theorem 5.4]{applications} for the proof.
\end{proof}

\subsection{Higher kernels and the main result}
Our aim is to characterize the exact category of the form $^\perp\WW =\XXX_\WW$ for some $n$-cotilting subcategory $\WW$ of a module category $\mod\CC$.
In this subsection \emph{we assume that $\EE$ is a skeletally small exact category with enough projectives $\PP$ and enough injectives $\II$}. Recall that we have the Morita embedding (\ref{moritaembedding}) $\P :\EE \to \XXX_\WW$ with $\WW = \P(\II)$ by Theorem \ref{pe}.

The following result gives a criterion for this embedding $\P: \EE \to \XXX_\WW$ to be dense up to summands.
A similar method was used in the proof of \cite[Theorem 2.7]{kal}. We simplify the proof by using the analogue of Auslander-Buchweitz approximation for exact categories, which we refer to Appendix A.
\begin{proposition}\label{xispe}
Let $\EE$ be a skeletally small exact category with enough projectives $\PP$  and enough injectives $\II$. Consider the Morita embedding $\P:\EE \to \mod\PP$ and put $\WW := \P(\II)$. If $\widehat{\P(\EE)}^n = \mod\PP$ holds for some $n \geq 0$, then $\WW$ is an $n$-cotilting subcategory of $\mod\PP$ and we have $\XXX_\WW = \add \P(\EE)$. In this case, $\XXX_\WW = \P(\EE)$ holds if and only if $\EE$ is idempotent complete.
\end{proposition}
\begin{proof}
Since we have $\P(\EE) \subset \XXX_\WW$ by Theorem \ref{pe}, the equality $\widehat{\P(\EE)}^n = \mod\PP$ implies $\widehat{\XXX_\WW}^n = \mod\PP$. Thus $\WW$ is $n$-cotilting by Proposition \ref{cotiltingcond}. On the other hand, recall that $\P(\EE)$ is preresolving in $\mod\PP$ and has enough injectives $\WW$. By Proposition \ref{xxww}, $\XXX_\WW$ is resolving in $\mod\PP$ and has enough injectives $\add\WW$. Applying Corollary \ref{icchi} to $\XX = \XXX_\WW$ and $\XX' = \P(\EE)$, we have $\XXX_\WW = \add\P(\EE)$.
The last assertion is clear.
\end{proof}

To investigate further properties of cotilting subcategories, we extend the notion of \emph{$n$-kernels} \cite{jasso} to $n \geq -1$, which simplify our results below.
\begin{definition}\label{nkernel}
Let $\CC$ be an additive category.
\begin{enumerate}[leftmargin=*]
 \item Let $n$ be an integer $n\geq 1$. We say that $\CC$ \emph{has $n$-kernels} if for any morphism $f:X\to Y$ in $\CC$, there exists a complex
 \[
 0 \to X_n \xrightarrow{f_n} \cdots \xrightarrow{f_2} X_1 \xrightarrow{f_1} X \xrightarrow{f} Y
 \]
  in $\CC$ such that the following diagram is exact.
  \[
 0 \to \CC(-,X_n) \xrightarrow{\CC(-,f_n)}  \cdots \xrightarrow{\CC(-,f_2)}  \CC(-,X_1) \xrightarrow{\CC(-,f_1)} \CC(-,X) \xrightarrow{\CC(-,f)}  \CC(-,Y)
  \]
 \item Suppose that $\CC$ is in addition an exact category.
 \begin{enumerate}
  \item We say that $\CC$ has \emph{$0$-kernels} if every morphism $f$ in $\CC$ can be factored as a deflation followed by a monomorphism, that is, $f=i g$ holds for a deflation $g$ and a monomorphism $i$. 
  \item We say that $\CC$ has \emph{$(-1)$-kernels} if every morphism can be factored as a deflation followed by an inflation.
 \end{enumerate}
\end{enumerate} 
Dually we define the notion of \emph{$n$-cokernels} for $n \geq -1$.
\end{definition}
Note that having $n$-kernels implies having $m$-kernels for $m \geq n \geq -1$. Also we point out that an exact category $\CC$ has $(-1)$-kernels if and only if $\CC$ is abelian and its exact structure is the usual exact structure on abelian categories.

We remark that having higher kernels is almost equivalent to the finiteness of the global dimension, as the following classical proposition shows.
\begin{proposition}\label{gldim}
Let $\CC$ be a skeletally small idempotent complete additive category with weak kernels and $n$ an integer $n \geq 1$. Then the following hold.
\begin{enumerate}
\item $\CC$ has $n$-kernels if and only if the global dimension of $\mod\CC$ is at most $n+1$.
\item Suppose that $\End_\CC(M)$ is right coherent for an object $M \in \CC$. Then $\add M$ has $n$-kernels if and only if the right global dimension of $\End_\CC(M)$ is at most $n+1$.
\end{enumerate}
\end{proposition}

The following proposition gives a necessary condition for $\EE$ to be equivalent to $\XXX_\WW$ for an $n$-cotilting subcategory $\WW$ of a module category $\mod\CC$.

\begin{proposition}\label{cotiltker}
Let $\CC$ be a skeletally small additive category with weak kernels and $\WW$ an $n$-cotilting subcategory of $\mod\CC$ for $n \geq 0$. Then $\XXX_\WW$ has $(n-1)$-kernels. 
\end{proposition}
\begin{proof}
Recall that $\XXX_\WW = {}^\perp\WW$ since $\WW$ is $n$-cotilting. Also note that $\mod\CC$ is abelian by Proposition \ref{wk} since $\CC$ has weak kernels.
If $n=0$, that is, $\XXX_\WW = \mod\CC$, then $\XXX_\WW$ clearly has $(-1)$-kernels since $\mod\CC$ is abelian. 

Let $f:X \to Y$ be a morphism in $\XXX_\WW$. Then we have conflations $M_1 \infl X \defl M_0$ and $M_0 \infl Y \defl M_{-1}$ in $\mod\CC$ such that $f$ is the composition $X \defl M_0 \infl Y$. 
Suppose that $n=1$. It follows that $\Ext^{>0}_{\mod\CC}(M_0,\WW) = \Ext^{>1}_{\mod\CC}(M_{-1},\WW)=0$ since $\id \WW \leq 1$. Thus $M_0$ is in $^\perp\WW=\XXX_\WW$. Since $^\perp\WW$ is resolving in $\mod\CC$, it follows that $M_1$ is also in $^\perp\WW$. Consequently $X\defl M_0$ is a deflation in $^\perp\WW$, which shows that $\XXX_\WW$ has $0$-kernels.

Next let us consider the case $n \geq 2$. Note that $\XXX_\WW$ is contravariantly finite in $\mod\CC$ by Corollary \ref{abcotilt}. This gives a right $\XXX_\WW$-approximation $X_1 \to M_1$. Since $\XXX_\WW$ contains all projective objects, this morphism is an epimorphism. Hence we obtain a conflation $M_2 \infl X_1 \defl M_1$ in $\mod\CC$. Repeat this construction until we get $M_{n-1} \infl X_{n-2} \defl M_{n-2}$. (If $n=2$, we interpret $X_0 = X$.) By the dimension shift argument, we have $\Ext^{>0}_{\mod\CC}(M_{n-1},\WW) = \cdots = \Ext^{>n}_{\mod\CC}(M_{-1},\WW)=0$, which implies $M_{n-1}\in\XXX_\WW$. Consider the complex 
\[
0 \to M_{n-1} \to X_{n-2} \to \cdots \to X_{1} \to X \to Y
\]
 in $\XXX_\WW$. Since the morphism $X_i \defl M_i$ in each conflation $M_{i+1} \infl X_i \defl M_i$  is a right $\XXX_\WW$-approximation for $i \geq 1$, it is easy to see that 
 \[
 0 \to \XXX_\WW(-,M_{n-1}) \to \cdots \to \XXX_\WW(-,X_1) \to \XXX_\WW(-,X) \to \XXX_\WW(-,Y)
 \] is exact. Hence $\XXX_\WW$ has $(n-1)$-kernels.
\end{proof}
Surprisingly, having $n$-kernels is not only necessary but also sufficient for an exact category to be equivalent to $\XXX_\WW$ for an $n$-cotilting subcategory $\WW$ of $\mod\CC$, as the following proposition shows. This was proved in \cite[Theorem 2.7]{kal} for the case of Frobenius categories.
\begin{proposition}\label{kercotilt}
Let $\EE$ be a skeletally small exact category with enough projectives $\PP$ and enough injectives $\II$. Consider the Morita embedding \emph{(\ref{moritaembedding})} $\P: \EE \to \mod\PP$. Suppose that there exists a subcategory $\MM$ of $\EE$ which contains $\PP$ and has $(n-1)$-kernels for $n \geq 0$. Then the following hold.
\begin{enumerate}
\item $\WW := \P(\II)$ is an $n$-cotilting subcategory of $\mod\PP$.
\item The equality $\add\P(\EE) = \XXX_\WW$ holds. Moreover $\P(\EE) = \XXX_\WW$ if and only if $\EE$ is idempotent complete.
\item $\PP$ has weak kernels.
\end{enumerate}
\end{proposition}
\begin{proof}
Since $\MM$ has weak kernels and every object in $\MM$ has a deflation from some object in $\PP$, it is easy to see that $\PP$ has weak kernels. Thus (3) holds.

By Proposition \ref{xispe}, we only have to show $\widehat{\P(\EE)}^n = \mod\PP$ to prove (1) and (2). Let $X$ be any object in $\mod\PP$. Fix a morphism $P_1 \to P_0$ in $\PP$ such that $\P P_1 \to \P P_0 \to X \to 0$ is exact.

First suppose that $n \geq 2$. Then by taking the $(n-1)$-kernel of $P_1 \to P_0$ inside $\MM$, we get a complex 
\[
0 \to M_n \to \cdots \to M_2 \to P_1 \to P_0
\]
in $\MM$. Since $\MM$ contains $\PP$, it easily follows that 
 \[
 0\to \P M_n \to \cdots \to \P P_1 \to \P P_0 \to X \to 0
 \]
is a complex which decomposes into conflations in $\mod\PP$. This gives $\widehat{\P(\EE)}^n = \mod\PP$.

The case $n=0$ and $1$ are quite similar and left to the reader.
\end{proof}
This immediately gives the following characterization of cotilting subcategories of a module category $\mod\CC$ amongst Wakamatsu tilting subcategories.
\begin{corollary}\label{cotiltwakamatsu}
Let $\CC$ be a skeletally small additive category and $\WW$ a Wakamatsu tilting subcategory of $\mod\CC$. Then the following are equivalent for $n \geq 0$.
\begin{enumerate}
\item $\CC$ has weak kernels and $\WW$ is $n$-cotilting.
\item $\XXX_\WW$ has $(n-1)$-kernels.
\item $\XXX_\WW$ has a subcategory $\MM$ which contains all projective objects and has $(n-1)$-kernels.
\end{enumerate}
\end{corollary}
\begin{proof}
(1) $\Rightarrow$ (2): This is Proposition \ref{cotiltker}.

(2) $\Rightarrow$ (3): This is trivial.

(3) $\Rightarrow$ (1): This is obvious from Proposition \ref{kercotilt}, since the Morita embedding $\P$ can be identified with the natural inclusion $\XXX_\WW \to \mod\CC$ in this case.
\end{proof}

As an application, we immediately get the following information about global dimensions.
\begin{corollary}
Let $\Lambda$ be an artin $R$-algebra and $W$ a Wakamatsu tilting $\Lambda$-module which is \emph{not} a cotilting module. Then for any $M \in \XXX_\WW$ such that $\Lambda \in \add M$, the global dimension of $\End_\Lambda(M)$ is infinite.
\end{corollary}
\begin{proof}
Let $W$ be a Wakamatsu tilting $\Lambda$-module and suppose that there exists a $\Lambda$-module $M \in \XXX_\WW$ such that $\Lambda \in \add M$ and the global dimension of $\End_\Lambda(M)$ is finite. Then by Proposition \ref{gldim}, $\add M$ has $n$-kernels for some $n \geq 0$. Applying Corollary \ref{cotiltwakamatsu} to $\MM = \add M$, we have that $M$ must be a cotilting module.
\end{proof}

Summarizing these results, we obtain the internal characterizations of exact categories associated with cotilting subcategories.
\begin{theorem}\label{cor1}
Let $\EE$ be a skeletally small exact category. For an integer $n \geq 0$, the following are equivalent.
\begin{enumerate}
\item There exist a skeletally small additive category $\CC$ with weak kernels and an $n$-cotilting subcategory $\WW$ of $\mod\CC$ such that $\EE$ is exact equivalent to $\XXX_\WW$.
\item $\EE$ is idempotent complete, has enough projectives and injectives and has $(n-1)$-kernels.
\item $\EE$ is idempotent complete, has enough projectives and injectives and $\EE$ has a subcategory $\MM$ such that $\MM$ contains $\PP(\EE)$ and $\MM$ has $(n-1)$-kernels.
\end{enumerate}
\end{theorem}

By restricting our attention to artin $R$-algebras, we have the following criterion.
\begin{corollary}\label{artincase}
Let $\EE$ be a skeletally small Hom-finite exact $R$-category. For an integer $n \geq 0$, the following are equivalent.
\begin{enumerate}
\item There exist an artin $R$-algebra $\Lambda$ and a cotilting $\Lambda$-module $U$ with $\id U \leq n$ such that $\EE$ is exact equivalent to $^\perp U$.
\item $\EE$ is idempotent complete, has a projective generator $P$ and has enough injectives, and has $(n-1)$-kernels.
\item $\EE$ is idempotent complete, has a projective generator $P$ and has enough injectives, and $\EE$ has a subcategory $\MM$ such that $\MM$ contains $P$ and $\MM$ has $(n-1)$-kernels.
\item $\EE^{\op}$ satisfies one of the conditions {\upshape (1)-(3)}.
\end{enumerate}
\end{corollary}
\begin{proof}
The conditions (1)-(3) are equivalent by Theorem \ref{cor1} since the notion of $n$-cotilting subcategories in $\mod\Lambda$ coincides with usual cotilting modules by Proposition \ref{wakamatsucotilting} for an artin $R$-algebra $\Lambda$. It remains to show (4). It suffices to check that (1) is self-dual. Suppose that $\EE={}^\perp U$ for a cotilting $\Lambda$-module with $\id U \leq n$. The Brenner-Butler theorem implies that $_\Gamma U$ is a cotilting $\Gamma$-module for $\Gamma := \End_\Lambda(U)$ and $\Hom_\Lambda(-,U)$ gives an exact duality $\EE = {}^\perp (U_\Lambda) \to {}^\perp( {}_\Gamma U) $. Thus $\EE^{\op}$ is exact equivalent to $^\perp( {}_\Gamma U)$, so $\EE^{\op}$ satisfies (1).
\end{proof}

Next we specialize these results to Frobenius categories. Recall that a two-sided noetherian ring $\Lambda$ is called \emph{Iwanaga-Gorenstein} if the right and left injective dimensions of $\Lambda$ are finite. It was shown in \cite[Lemma A]{za} that $\id {}_\Lambda \Lambda = \id \Lambda_\Lambda$ holds for an Iwanaga-Gorenstein ring $\Lambda$. We call $\Lambda$ an \emph{$n$-Iwanaga-Gorenstein} if $\id \Lambda_\Lambda \leq n$. If $\Lambda$ is an artin $R$-algebra, then $\Lambda$ is Iwanaga-Gorenstein if and only if $\Lambda_\Lambda$ is a cotilting module, or $\widehat{\GP\Lambda}=\mod\Lambda$ (see \cite[Proposition 6.1]{applications} or \cite[Proposition 4.2]{cm}).

We have the following characterization of $\GP \Lambda$ for an Iwanaga-Gorenstein ring $\Lambda$. Note that this result was mentioned in \cite[Proposition 4]{kal2}.
\begin{corollary}
Let $\EE$ be a skeletally small exact category. For an integer $n \geq 0$, the following are equivalent.
\begin{enumerate}
\item There exists an $n$-Iwanaga-Gorenstein ring $\Lambda$ such that $\EE$ is exact equivalent to $\GP \Lambda$.
\item $\EE$ is idempotent complete and Frobenius, has a projective generator $P$ such that $\End_\EE(M)$ is noetherian, and has both $(n-1)$-kernels and $(n-1)$-cokernels.
\item $\EE$ is idempotent complete and Frobenius, has a projective generator $P$ such that $\End_\EE(P)$ is noetherian, and has a subcategory $\MM$ such that $\MM$ contains $P$ and $\MM$ has both $(n-1)$-kernels and $(n-1)$-cokernels. 
\end{enumerate}
\end{corollary}
\begin{proof}
(1) $\Rightarrow$ (2):
We may assume that $\EE=\GP \Lambda$ for an Iwanaga-Gorenstein ring $\Lambda$. By Theorem \ref{cor1}, $\EE$ has $(n-1)$-kernels. To show that $\EE$ has $(n-1)$-cokernels, we only have to note that $\Hom_\Lambda(-,\Lambda):\GP \Lambda \to \GP \Lambda^{\op}$ gives a duality.

(2) $\Rightarrow$ (3): This is trivial.

(3) $\Rightarrow$ (1): 
Theorem \ref{cor1} gives $\id \Lambda_\Lambda \leq n$ and $\EE \equi \GP\Lambda$. Since (3) is self-dual, we have $\id {}_\Lambda\Lambda \leq n$, which implies $\Lambda$ is $n$-Iwanaga-Gorenstein.
\end{proof}

\section{Applications to artin $R$-algebras}
In the previous section, we gave a necessary and sufficient condition for an exact category $\EE$ to be exact equivalent to $^\perp U$ for some cotilting module $U$. As we have seen in Theorem \ref{cor1}, it suffices to check that $\EE$ has enough projectives and injectives, and has higher kernels. In this section, we apply this criterion to the representation theory of artin algebras. \emph{We always denote by $\Lambda$ an artin $R$-algebra in this section.}

\subsection{The category $(\mod\Lambda) / [\Sub M]$ as a torsionfree class}
Recall that a subcategory $\FF$ of $\mod\Lambda$ is said to be a \emph{torsionfree class} if $\FF$ is closed under extensions and submodules.
It is classical that for a cotilting module $U \in \mod\Lambda$ with $\id U \leq 1$, the subcategory $^\perp U$ of $\mod\Lambda$ is a torsionfree class, see \cite{ASS}.

To state our main theorem in this subsection, let us recall the definition of ideal quotients of additive categories. Let $\CC$ be an additive category and $\DD$ an additive subcategory of $\CC$. Denote by $[\DD]$ the ideal of $\CC$ consisting of all morphisms in $\CC$ which factor through some objects in $\DD$. Then the ideal quotient $\CC/[\DD]$ is an additive category whose objects are those in $\CC$ and whose morphisms are given by 
\[
(\CC/[\DD])(M,N) := \CC(M,N) /[\DD](M,N)
\]
 for $M$ and $N$ in $\CC$. Also see Definition \ref{functfin} for the notion of functorially finiteness.
\begin{theorem}\label{main1}
Let $\Lambda$ be an artin $R$-algebra and $\CC$ a functorially finite subcategory of $\mod\Lambda$ which is closed under submodules. Put $\EE := \mod\Lambda$ and denote by $\pi:\EE \defl \EE/[\CC]$ the natural quotient functor. Then the following hold.
\begin{enumerate}
\item There exists an exact structure on $\EE/[\CC]$ induced by $\CC$-conflations (see Appendix B for details) such that $\EE/[\CC]$ has $0$-kernels.
\item The exact category $\EE/[\CC]$ has enough projectives $\PP$ and injectives $\II$.
\item $\PP$ (resp. $\II$) is precisely the essential image of $\add ( \Lambda, \tau^- \CC )$ (resp. $\add ( D\Lambda, \tau \CC )$) under $\pi$.
\item $\PP$ is  of finite type if and only if $\II$ is of finite type. 
\item Suppose that {\upshape (4)} holds. Let $M$ (resp. $N$) be an object satisfying $\PP = \add M$ (resp. $\II = \add N$) and put $\Gamma := \End_{\EE/[\CC]}(M)$. Then $(\EE/[\CC])(M,N)$ is a $1$-cotilting $\Gamma$-module and $\EE/[\CC]$ is exact equivalent to a torsionfree class $^\perp U$. 
\end{enumerate}
\end{theorem}
We omit the dual result for the case when $\CC$ is closed under quotients.
Note that this was proved in \cite[Proposition 5.5]{tau4} in case $\CC = \mod (\Lambda / I)$ for some two-sided ideal $I$ and $(\mod\Lambda)/[\CC]$ is of finite type.
\begin{remark}
Let $\CC$ be a subcategory of $\mod\Lambda$ which is closed under submodules. By \cite[Proposition 4.7]{as}, the following are equivalent.
\begin{enumerate}
\item $\CC$ is functorially finite.
\item $\CC$ is contravariantly finite.
\item $\CC = \Sub M$ for some $M \in \mod\Lambda$.
\end{enumerate}
Here $\Sub M$ is the subcategory consisting of all modules cogenerated by $M$, in other words, the smallest additive subcategory containing $M$ which is closed under submodules.
\end{remark}

To begin the proof, the following elementary observation is useful. In what follows, we denote by $\pi(f) : \underline{L} \to \underline{M}$ in $\EE/[\CC]$ the image of $f:L \to M$ in $\EE$ under the natural functor $\pi:\EE \defl \EE/[\CC]$.

\begin{proposition}\label{mono}
Let $\Lambda$ be a right noetherian ring. The following are equivalent for an additive subcategory $\CC$ of $\mod\Lambda$.
\begin{enumerate}
\item $\CC$ is closed under quotient modules.
\item For every $X$ in $\mod\Lambda$, there exists a right $\CC$-approximation $C_X \infl X$ which is an injection.
\item $\CC$ is contravariantly finite, and epimorphisms are preserved by the natural functor $\pi:\mod\Lambda \defl (\mod \Lambda) / [\CC]$. 
\end{enumerate}
\end{proposition}
\begin{proof}
(1) $\Rightarrow$ (2):
Let $X$ be in $\mod\Lambda$. Consider the direct sum $\bigoplus_{C,f} C \to X$ where $C$ runs over isomorphism classes of objects in $\CC$ and $f$ runs over all morphisms $f:C \to X$. Write $C_X$ for its image. Since $\Lambda$ is assumed to be right noetherian, $C_X$ is a quotient module of finite direct sum of objects in $\CC$, which implies that $C_X$ is also in $\CC$. It is easy to see that $C_X \infl X$ is a right $\CC$-approximation.

(2) $\Rightarrow$ (3):
Contravariantly finiteness is obvious.
Let $0 \to L \xrightarrow{f} M \xrightarrow{g}  N \to 0$ be a short exact sequence in $\mod\Lambda$. We will see that $\pi(g)$ is epic. So let $\varphi:N \to X$ be a morphism in $\mod\Lambda$ such that $\pi (\varphi g)=\pi(\varphi) \pi(g)=0$ in $(\mod\Lambda)/[\CC]$. 
By the definition of the quotient category, it follows that $\varphi g$ factors through some object in $\CC$. It follows that $\varphi g$ factors through the right $\CC$-approximation $i:C_X \infl X$. Thus there exists a morphism $h:M \to C_X$ which makes the following diagram commute.
\[
\xymatrix{
 0 \ar[r] & L \ar[r]^{f} & M \ar[r]^{g}\ar[d]_{h} & N \ar[r]\ar[d]^{\varphi} & 0 \\
 & & C_X \inflr^{i}  & X & 
 }
\]
Since $i h f = \varphi g f =0$ and $i$ is monic, we have $h f=0$. Thus there exists a morphism $k:N \to C_X$ such that $h = k g$. Then $i k g = i h = \varphi g$ holds, which implies that $i k = \varphi$ since $g$ is epic. Therefore $\pi(\varphi)=0$ holds in $\EE/[\CC]$, which proves that $\pi(g)$ is epic.

(3) $\Rightarrow$ (1):
Suppose that $C$ is in $\CC$ and $f:C \defl X$ is a surjection. It follows from (3) that $\pi(f):\u{C} \to \u{X}$ is epimorphism in $\EE/[\CC]$. However $\u{C}$ is a zero object in $\EE/[\CC]$, thus $\u{X} \iso 0$, that is, $X \in \CC$.
\end{proof}
Note that if $\Lambda$ is a right artinian ring, then the dual of Proposition \ref{mono} holds.

To prove Theorem \ref{main1}, we have to introduce a new exact structure given by $\CC$-conflations on $\mod\Lambda$, which induces the desired exact structure on $(\mod\Lambda)/[\CC]$. 
For the notion of $(-,\CC)$-conflations, $(\CC,-)$-conflations and $\CC$-conflations, we refer the reader to Appendix B.

Now let us prove Theorem \ref{main1}. We divide the proof into several propositions. \emph{For the rest of this subsection, we always assume that $\Lambda$ is an artin $R$-algebra and that $\CC$ is a functorially finite subcategory of $\EE:=\mod\Lambda$ which is closed under submodules.}
\begin{proposition}\label{proof1}
Endow $\EE$ with an exact structure $\EE_\CC$. Then $\EE/[\CC]$ naturally inherits an exact structure from $\EE_\CC$ as in Proposition \ref{di}.
\end{proposition}
\begin{proof}
Denote by $\pi:\EE \defl \EE/[\CC]$ the natural functor. According to Proposition \ref{di}, it suffices to show that the image of every $\CC$-inflation under $\pi$ is a monomorphism and the image of every $\CC$-deflation is an epimorphism.
The assertion for $\CC$-inflations follows from the dual of Proposition \ref{mono}.

Let $0 \to L \to M \to N \to 0$ be a $\CC$-conflation. We show that $\u{M} \to \u{N}$ is an epimorphism in $\EE/[\CC]$. Suppose that $N \to X$ is a morphism in $\EE$ such that $\u{M} \to \u{N} \to \u{X}$ is zero. Then the composition $M \to N \to X$ factors through the left $\CC$-approximation $M \to C^M$. This yields the following commutative diagram
\[
\xymatrix{
0 \ar[r] & L \ar[r] \ar@{.>}[d] & M \ar[r]\ar[d] & N \ar[r]\ar[d] & 0 \\
0 \ar[r] & K \ar[r] & C^M \ar[r] & X &
}
\]
with exact rows. Since $\CC$ is closed under submodules, $K$ is in $\CC$. Because $L \infl M \defl N$ is $(-,\CC)$-conflation, $L \to K$ factors through $L \to M$. Then it is routine to check that $N \to X$ factors through $C^M\to X$, thus $\u{N} \to \u{X}$ is zero in $\EE/[\CC]$ as desired. 
\end{proof}

\begin{proposition}\label{proof2}
The exact category $\EE/[\CC]$ has $0$-kernels.
\end{proposition}
\begin{proof}
We have to show that every morphism $\pi(f):\underline{X} \to \underline{Y}$ in $\EE/[\CC]$ can be factored as a $\CC$-deflation followed by a monomorphism.

Let $f:X \to Y$ be a morphism in $\EE=\mod\Lambda$. Then we have the factorization $X \defl \im f \infl Y$. Note that $\underline{\im f} \to \underline{Y}$ is monic by the dual of Proposition \ref{mono}. Thus it suffices to show that $\u{X} \to \u{\im f}$ can be factored as a $\CC$-deflation followed by a monomorphism. Therefore we may assume that $f$ is a surjection.

Let $C_Y \to Y$ be a right $\CC$-approximation. Consider the pullback diagram
\[
\xymatrix{
0 \ar[r] & K \ar@{=}[d] \ar[r] & E \ar[r] \ar[d] & C_Y \ar[d] \ar[r] & 0\\
0 \ar[r] & K \ar[r] & X \ar[r] & Y \ar[r] & 0
}
\]
in $\EE$. Then the right square gives a short exact sequence
\[
0 \to E \to X \oplus C_Y \to Y \to 0
\]
in $\EE$. Since $C_Y \to Y$ is a right $\CC$-approximation, clearly this sequence is a $(\CC,-)$-conflation. To transform it into a $\CC$-conflation, take a left $\CC$-approximation $E \defl C^E$ and consider the pushout
\begin{eqnarray}\label{goccha}
\xymatrix{
0 \ar[r] & E \defld \ar[r] & X \oplus C_Y \defld^a \ar[r]^{c a} & Y \ar[r] \ar@{=}[d] & 0\\
0 \ar[r] & C^E \ar[r]^b & F \ar[r]^c & Y \ar[r] & 0
}
\end{eqnarray}
in $\EE$. The left square gives a short exact sequence
\begin{equation}\label{goccha2}
0 \to E \to C^E \oplus X \oplus C_Y \to F \to 0
\end{equation}
in $\EE$, which is a $(-,\CC)$-conflation since $E\defl C^E$ is a left $\CC$-approximation. Since $c a$ is a $(\CC,-)$-deflation, a diagram chase for (\ref{goccha}) shows that (\ref{goccha2}) is also a $(\CC,-)$-conflation. Thus (\ref{goccha2}) is a $\CC$-conflation. Because $\underline{C^E}$ and $\underline{C_Y}$ is zero in $\EE/[\CC]$, we obtain a deflation $\pi(g): \underline{X} \defl \underline{F}$ in $\EE_\CC/[\CC]$, where we write $a = [g,h]$.

By the commutativity of the right square in (\ref{goccha}), we have $\pi(f) = \pi(c) \pi(g)$.
Since $\pi(g)$ is a deflation, it suffices to show that $\pi(c):\underline{F} \to \underline{Y}$ is monic in $\EE/[\CC]$.

Let $\varphi:Z\to F$ be a morphism such that the composition $c \varphi$ factors through some objects in $\CC$. Since $C_Y \to Y$ is a right $\CC$-approximation, it follows that $c \varphi$ factors through $C_Y \to Y$. Thus we have a map $d:Z \to C_Y$ such that $c \varphi = c a \circ{}^t [0,d]$. This implies that there exists a map $e:Z \to C^E$ such that $\varphi - a \circ{}^t [0,d] = b e$, hence $\varphi = a \circ{}^t [0,d] + b e$. Since the images of $^t [0,d]$ and $b$ under $\pi$ are zero in $\EE/[\CC]$, it follows that $\pi(\varphi)=0$ in $\EE/[\CC]$, which shows that $\pi(c)$ is monic.
\end{proof}
Now we are in position to finish the proof of Theorem \ref{main1}.
\begin{proof}[Proof of Theorem \ref{main1}]
(1): The assertion follows from Proposition \ref{proof1} and \ref{proof2}.

(2):
We basically follow the idea of \cite{as1}.
It is well-known that a short exact sequence $0 \to L \to M \to N \to 0$ is $(-,X)$-exact if and only if $(\tau^- X,-)$-exact (see, e.g. \cite[Corollary 4.4]{ARS}). Hence this sequence is $\CC$-exact if and only if it is $(\add (\CC, \tau^- \CC),-)$-exact. It was shown in \cite[Theorem 1.14]{as1} that this exact structure has enough projectives and injectives if and only if $\add (\CC, \tau^- \CC)$ is functorially finite. Since $\CC$ is functorially finite in $\mod\Lambda$, it follows that so is $\tau^- \CC$. Thus it is clear that $\add (\CC, \tau^- \CC)$ is functorially finite. Therefore $\EE_\CC$ has enough projectives and injectives. By Theorem \ref{di}, $\EE_\CC/[\CC]$ also has enough projectives and injectives, thus (2) holds.

(3):
Note that projective (resp. injective) objects in $\EE_\CC$ are precisely objects in $\add ( \Lambda, \CC, \tau^- \CC )$ (resp. $\add ( D\Lambda, \CC, \tau \CC )$). Thus Theorem \ref{di} immediately gives (3).

(4):
This is clear from (3).

(5):
By using the Morita embedding $\Hom_{\EE/[\CC]}(M,-):\EE/[\CC] \to \mod\Gamma$, the assertion is immediate from Corollary \ref{artincase}.
\end{proof}

As an immediate consequence, we have the following.
\begin{corollary}
Let $\Lambda$ be an artin $R$-algebra and $S$ a semisimple module in $\mod\Lambda$. Put $\EE := (\mod \Lambda) /[\add S]$ and $\Gamma := \End_\EE (\Lambda \oplus \tau^- S)$. Then $U := \EE(\Lambda \oplus \tau^- S, D \Lambda \oplus \tau S) \in \mod\Gamma$ is a $1$-cotilting module and  we have an equivalence
\begin{equation}\label{apr}
\EE(\Lambda \oplus \tau^- S,-): \EE \equi {}^\perp U.
\end{equation}
\end{corollary}
If $S$ is simple projective and not injective, then $\EE$ is shown to be equivalent to the subcategory of $\mod\Lambda$ consisting of modules $M$ such that $\Hom_\Lambda(M,S)=0$. In this case, $\Gamma$ is an APR tilt of $\Lambda$ at $S$ and $U$ is the corresponding APR cotilting $\Gamma$-module (see \cite[Example 2.8(c)]{ASS}) and the above equivalence (\ref{apr}) coincides with the classical equivalence \cite[Theorem 3.8]{ASS}.

Historically, the notion of APR tilts was generalized to tilting modules, which induce an embedding of a \emph{subcategory} of $\mod\Lambda$ into another module category by the Brenner-Butler theorem. On the other hand, this corollary (and our main theorem) gives an embedding of a certain \emph{quotient category} of $\mod\Lambda$ into another module category. Thus our results can be regarded as another generalization of APR tilts.

\begin{example}
Let $\Lambda$ be a self-injective Nakayama algebra of Loewy length four given by the cyclic quiver with three vertices over a field $k$. Let $S$ be a simple module corresponding the white dot in the following Auslander-Reiten quiver of $\mod \Lambda$
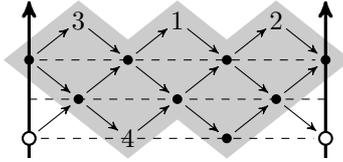
\[ \begin{tikzpicture}[scale=.65,yscale=-.8]
 \fill [fill1] (-0.5,1) -- (1,-0.5) -- (2,0.5) -- (3,-0.5) -- (4,0.5) -- (5,-0.5) -- (6.5,1) -- (4,3.5) -- (3,2.5) -- (2,3.5) -- cycle;

 \draw [->, very thick] (0,3.5) -- (0,-0.5);
 \draw [->, very thick] (6,3.5) -- (6,-0.5);
 
 \node (213) at (0,1) [vertex] {};
 \node (1) at (0,3) [cvertex] {};
 \node (3213) at (1,0) [tvertex] {3};
 \node (21) at (1,2) [vertex] {};
 \node (321) at (2,1) [vertex] {};
 \node (2) at (2,3) [tvertex] {4};
 \node (1321) at (3,0) [tvertex] {1};
 \node (32) at (3,2) [vertex] {};
 \node (132) at (4,1) [vertex] {};
 \node (3) at (4,3) [vertex] {};
 \node (2132) at (5,0) [tvertex] {2};
 \node (13) at (5,2) [vertex] {};
 \node (n213) at (6,1) [vertex] {};
 \node (n1) at (6,3) [cvertex] {};
 
 \draw [dashed] (1) --  (2);
 \draw [dashed] (2) -- (n1);
 \draw [dashed] (0,2) -- (6,2);
 \draw [dashed] (213) -- (n213);
 
 \draw [->] (1) -- (21);
 \draw [->] (21) -- (2);
 \draw [->] (2) -- (32);
 \draw [->] (32) -- (3);
 \draw [->] (3) -- (13);
 \draw [->] (13) -- (n1);
 \draw [->] (213) -- (21);
 \draw [->] (21) -- (321);
 \draw [->] (321) -- (32);
 \draw [->] (32) -- (132);
 \draw [->] (132) -- (13);
 \draw [->] (13) -- (n213);
 \draw [->] (213) -- (3213);
 \draw [->] (3213) -- (321);
 \draw [->] (321) -- (1321);
 \draw [->] (1321) -- (132);
 \draw [->] (132) -- (2132);
 \draw [->] (2132) -- (n213);

\end{tikzpicture} \]
where we identify two vertical arrows. Then obviously $\CC := \add S$ is closed under submodules. The full translation subquiver consisting of the shaded areas gives the Auslander-Reiten quiver of $(\mod \Lambda)/[\CC]$.

Then $M$ in Theorem \ref{main1} corresponds to the numbered vertices and $\Gamma$ is given by the quiver
\[
\xymatrix{
1 \ar[rd]_{\delta} & 2 \ar[l]^{\varepsilon} \ar@/^/[r]^{\alpha} & 4 \ar@/^/[l]^{\beta} \\
 & 3\ar[u]_{\gamma} & &
 }
 \]
and relation $\beta \alpha = \gamma \alpha = \delta \gamma \varepsilon = \beta \varepsilon = \alpha \beta - \varepsilon \delta \gamma =0$. The following diagram gives the Auslander-Reiten quiver of $\mod\Gamma$
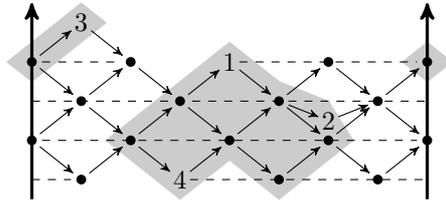
\[ \begin{tikzpicture}[scale=.65,yscale=-.8]
 \fill [fill1] (-0.5,1) -- (1,-0.5) -- (1.5,0) -- (0,1.5) -- cycle;
 \fill [fill1] (1.5,3) -- (4,0.5) -- (5,1.5) -- (6,2) -- (6.5,3) -- (5,4.5) -- (4,3.5) -- (3,4.5) -- cycle;
 \fill [fill1] (7.5,1) -- (8,0.5) -- (8.5,1) -- (8,1.5) -- cycle;
 \draw [->, very thick] (0,4.5) -- (0,-0.5);
 \draw [->, very thick] (8,4.5) -- (8,-0.5);
 
 \node (213) at (0,1) [vertex] {};
 \node (214) at (0,3) [vertex] {};
 \node (3213) at (1,0) [tvertex] {3};
 \node (21) at (1,2) [vertex] {};
 \node (24) at (1,4) [vertex] {};
 \node (321) at (2,1) [vertex] {};
 \node (2) at (2,3) [vertex] {};
 \node (32) at (3,2) [vertex] {};
 \node (42) at (3,4) [tvertex] {4};
 \node (132) at (4,1) [tvertex] {1};
 \node (342) at (4,3) [vertex] {};
 \node (1342) at (5,2) [vertex] {};
 \node (3) at (5,4) [vertex] {};
 \node (4) at (6,1) [vertex] {};
 \node (21342) at (6,2.5) [tvertex] {2};
 \node (13) at (6,3) [vertex] {};
 \node (2134) at (7,2) [vertex] {};
 \node (1) at (7,4) [vertex] {};
 \node (n213) at (8,1) [vertex] {};
 \node (n214) at (8,3) [vertex] {};

 \draw [dashed] (0,4) -- (24);
 \draw [dashed] (42) -- (8,4);
 \draw [dashed] (214) -- (n214);
 \draw [dashed] (0,2) -- (8,2);
 \draw [dashed] (213) -- (321);
 \draw [dashed] (132) -- (n213);

 \draw [->] (1) -- (n214);
 \draw [->] (214) -- (24);
 \draw [->] (24) -- (2);
 \draw [->] (2) -- (42);
 \draw [->] (42) -- (342);
 \draw [->] (342) -- (3);
 \draw [->] (3) -- (13);
 \draw [->] (13) -- (1);
 \draw [->] (2134) -- (n214);
 \draw [->] (214) -- (21);
 \draw [->] (21) -- (2);
 \draw [->] (2) -- (32);
 \draw [->] (32) -- (342);
 \draw [->] (342) -- (1342);
 \draw [->] (1342) -- (13);
 \draw [->] (13) -- (2134);
 \draw [->] (2134) -- (n213);
 \draw [->] (213) -- (21);
 \draw [->] (21) -- (321);
 \draw [->] (321) -- (32);
 \draw [->] (32) -- (132);
 \draw [->] (132) -- (1342);
 \draw [->] (1342) -- (4);
 \draw [->] (4) -- (2134);
 \draw [->] (213) -- (3213);
 \draw [->] (3213) -- (321);
 \draw [->] (1342) -- (21342);
 \draw [->] (21342) -- (2134);

\end{tikzpicture} \]
where we identify two vertical arrows. The shaded area corresponds to the subcategory $^\perp U$ of $\mod\Gamma$ in Theorem \ref{main1}. The above two shaded regions illustrate the equivalence $(\mod\Lambda)/[\CC] \equi {}^\perp U$.
\end{example}

\begin{remark}
We point out that the assumption \emph{$\CC$ is closed under submodules} in Theorem \ref{main1} cannot be replaced by \emph{$\CC$ is closed under images}. The simple example is as follows. Let $\Lambda$ be a path algebra of the quiver $1 \leftarrow 2 \leftarrow 3 \leftarrow 4 \leftarrow 5$ over a field $k$. Put $M := P(4)/S(1)$ where $P(4)$ is the indecomposable projective module corresponding to $4$ and $S(1)$ is the simple module corresponding to $1$. Then $\CC := \add M$ is easily checked to be closed under images. On the other hand, one can easily check that $(\mod \Lambda)/[\CC]$ does not have $0$-kernels, e.g. by using Proposition \ref{gldim}.
\end{remark}

\subsection{The category $\mod\Lambda$ as an exact category}
We assume that $\Lambda$ is an artin $R$-algebra. The category $\mod\Lambda$ has various exact structures given in Definition \ref{various}. An explicit classification was given in \cite[Proposition 3.3.2]{bu}.

It is natural to ask which exact structure has enough projectives or which has a projective generator. The answer of this was essentially given in \cite{as1} in terms of relative homological algebra. 
We call a subcategory $\MM$ of $\mod\Lambda$ a \emph{generating subcategory} if $\proj\Lambda \subset \MM = \add\MM$ holds. Dually we define a \emph{cogenerating subcategory}. A $\Lambda$-module $M$ is called a \emph{generator} (resp. \emph{cogenerator}) if $\add M$ is generating (resp. cogenerating).
\begin{proposition}\label{classical}
Let $\Lambda$ be an artin $R$-algebra.
\begin{enumerate}
\item There exists a bijection between exact structures on $\mod\Lambda$ with enough projectives and contravariantly finite generating subcategories of $\mod\Lambda$. It is given by sending an exact structure on $\mod\Lambda$ to the category of all projective objects, and the inverse map is given by sending a generating subcategory $\MM$ to $(\mod\Lambda)_{(\MM,-)}$.
\item The exact structure on $\mod\Lambda$ has a projective generator if and only if it has an injective cogenerator. If $G$ is a projective generator in this exact structure, then $C:= D \Lambda \oplus \tau M$ is an injective cogenerator.
\item The bijection of {\upshape (1)} restricts to a bijection between exact structures with projective generators and isomorphism classes of basic generators of $\mod\Lambda$.
\end{enumerate}
\end{proposition}
\begin{proof}
This follows directly from results in \cite[Theorem 1.15]{as1} and \cite[Proposition 1.7]{drss} (or Theorem \ref{modify}).
\end{proof}
Hence the nontrivial generators yield nontrivial embeddings of $\mod\Lambda$ into another module category.
Our main results in this direction is the following.
This theorem was essentially proved in \cite[Proposition 3.26]{as2} by using the method of relative cotilting theory. We give an alternative proof of this theorem by using our results. 
\begin{theorem}\label{main2}
Let $\Lambda$ be an artin $R$-algebra and $G$ be a generator of $\mod\Lambda$. Put $C:= D\Lambda \oplus \tau G$, $\Gamma := \End_\Lambda(G)$ and $U := \Hom_\Lambda(G,C) \in \mod\Gamma$. Then the following hold.
\begin{enumerate}
\item $U$ is a cotilting $\Gamma$-module with $\id U =2$ or $0$. The case $\id U=0$ occurs only when $G$ is projective.
\item $\Hom_\Lambda(G,-):\mod\Lambda \to \mod\Gamma$ induces an equivalence $\mod\Lambda \equi {}^\perp U$.
\item $\mod\Lambda$ admits an exact structure such that projective objects are precisely objects in $\add G$ and the equivalence $\mod \Lambda \equi {}^\perp U$ is an exact equivalence.
\item $\End_\Lambda(G)$ and $\End_\Lambda(C)$ are derived equivalent.
\end{enumerate}
\end{theorem}
We need the following preparation.
\begin{lemma}\label{below}
Let $\AA$ be an abelian category. Suppose that $\AA$ is endowed with some exact structure, which we denote by $\AA_\EE$. If $\AA_\EE$ has $0$-kernels, then it coincides with the standard exact structure on $\AA$.
\end{lemma}
\begin{proof}
It suffices to show that any epimorphism in $\AA_\EE$ is a deflation. 
Let $f:Y \to Z$ be an epimorphism. Since $\AA_\EE$ has $0$-kernels, there is a factorization $f = i g$ such that $g$ is a deflation and $i$ is a monomorphism. Since $f$ is an epimorphism, so is $i$. Thus $i$ is an isomorphism since $\AA$ is abelian. Therefore $f$ is a deflation.
\end{proof}
\begin{proof}[Proof of Theorem \ref{main2}]
Let $G$ be a generator and put $\EE := \mod\Lambda$. Then by Theorem \ref{classical}, $\EE_{(G,-)}$ has a projective generator $G$ and an injective cogenerator $C$. 

To use Corollary \ref{artincase}, we observe that $\EE_{(G,-)}$ has $1$-kernels, which is immediate since $\mod\Lambda$ has kernels.
Hence (2)-(4) hold by Corollary \ref{artincase}. The remaining statement of (1) easily follows from Lemma \ref{below}.
\end{proof}

As an application of Theorem \ref{main2}, we obtain the following result, which says that every module category of an artin $R$-algebra sits inside the module category of an artin $R$-algebra with finite global dimension, with a little modification of its exact structure.
\begin{corollary}
Let $\Lambda$ be an artin $R$-algebra. Then there exist an artin $R$-algebra $\Gamma$ with finite global dimension and a cotilting $\Gamma$-module $U$ with $\id U = 2$ or $0$ such that $\mod\Lambda \equi {}^\perp U$.
\end{corollary}
\begin{proof}
There exists a generator $G$ of $\mod\Lambda$ whose endomorphism ring has a finite global dimension, see \cite[Section 3]{repdim}. Thus the assertion follows from Theorem \ref{main2}.
\end{proof}

\begin{example}
Let $\Lambda$ be a path algebra of the quiver $1 \leftarrow 2 \leftarrow 3 \leftarrow 4$ over a field $k$. The following diagram is the Auslander-Reiten quiver of $\mod \Lambda$.
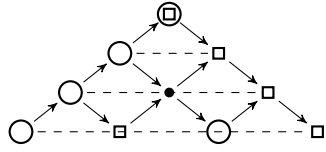
\[ \begin{tikzpicture}[scale=.65,yscale=-.8]
 \node (1) at (0,3) [cver] {};
 \node (2) at (1,2) [cver] {};
 \node (3) at (2,1) [cver] {};
 \node (4) at (2,3) [rver] {};
 \node (55) at (3,0) [cver] {};
 \node (5) at (3,0) [rver] {};
 \node (6) at (3,2) [vertex] {};
 \node (7) at (4,1) [rver] {};
 \node (8) at (4,3) [cver] {};
 \node (9) at (5,2) [rver] {};
 \node (10) at (6,3) [rver] {};
 
 \draw [dashed] (1) -- (10);
 \draw [dashed] (2) -- (9);
 \draw [dashed] (3) -- (7);
 
 \draw [->] (1) -- (2);
 \draw [->] (2) -- (3);
 \draw [->] (2) -- (4);
 \draw [->] (3) -- (5);
 \draw [->] (3) -- (6);
 \draw [->] (4) -- (6);
 \draw [->] (5) -- (7);
 \draw [->] (6) -- (7);
 \draw [->] (6) -- (8);
 \draw [->] (7) -- (9);
 \draw [->] (8) -- (9);
 \draw [->] (9) -- (10);

\end{tikzpicture} \]
Let $G$ be a generator of $\mod \Lambda$ corresponding to the circles. Then the associated cogenerator $C$ is the module indicated by the rectangles. Then $\Gamma := \End_\Lambda (G)$ is given by the quiver 
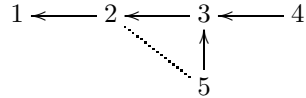
\[\xymatrix{1 & 2 \ar[l] & 3 \ar[l] & 4 \ar[l] \\ & & 5 \ar[u] \ar@{.}[ul] &}\]
where the dotted line indicates the zero relation. The Auslander-Reiten quiver of $\mod \Gamma$ is given by the following diagram.
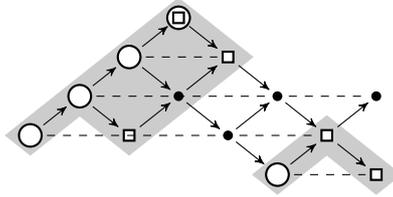
\[ \begin{tikzpicture}[scale=.65,yscale=-.8]

 \fill [fill1] (4.5,4) -- (6,2.5) -- (7.5,4) -- (7,4.5) -- (6,3.5) -- (5,4.5) -- cycle;
 \fill [fill1] (-.5,3) -- (3,-0.5) -- (4.5,1) -- (2,3.5) -- (1,2.5) -- (0,3.5) -- cycle; 
 
 \node (1) at (0,3) [cver] {};
 \node (21) at (1,2) [cver] {};
 \node (321) at (2,1) [cver] {};
 \node (2) at (2,3) [rver] {};
 \node (43211) at (3,0) [cver] {};
 \node (4321) at (3,0) [rver] {};
 \node (32) at (3,2) [vertex] {};
 \node (432) at (4,1) [rver] {};
 \node (3) at (4,3) [vertex] {};
 \node (43) at (5,2) [vertex] {};
 \node (53) at (5,4) [cver] {};
 \node (453) at (6,3) [rver] {};
 \node (5) at (7,2) [vertex] {};
 \node (4) at (7,4) [rver] {};
 
 \draw [dashed] (1) -- (453);
 \draw [dashed] (21) -- (5);
 \draw [dashed] (321) -- (432);
 \draw [dashed] (53) -- (4);
 
 \draw [->] (1) -- (21);
 \draw [->] (21) -- (2);
 \draw [->] (21) -- (321);
 \draw [->] (2) -- (32);
 \draw [->] (321) -- (4321);
 \draw [->] (321) -- (32);
 \draw [->] (4321) -- (432);
 \draw [->] (32) -- (432);
 \draw [->] (32) -- (3);
 \draw [->] (432) -- (43);
 \draw [->] (3) -- (43);
 \draw [->] (3) -- (53);
 \draw [->] (43) -- (453);
 \draw [->] (53) -- (453);
 \draw [->] (453) -- (5);
 \draw [->] (453) -- (4);

\end{tikzpicture} \]
The shaded area corresponds to the essential image of the embedding $\mod\Lambda \to \mod\Gamma$ in Theorem \ref{main2}. We keep the shapes of vertices in the quiver. In particular, the direct sum of all rectangles gives the 2-cotilting module $U$ and $\mod\Lambda$ is equivalent to $^\perp U$.
\end{example}

\appendix
\section{The Auslander-Buchweitz theory for exact categories} 
In this appendix, we shall study the Auslander-Buchweitz approximation theory, developed in \cite{ab}, in the context of exact categories. This is a useful tool to investigate cotilting subcategories. Our results in this appendix are used in Section 4.

First let us recall the following important notions.
\begin{definition}\label{functfin}
Let $\CC$ be an additive category and $\DD$ an additive subcategory of $\CC$.
\begin{enumerate}
 \item A morphism $f:D_X \to X$ in $\CC$ is said to be a \emph{right $\DD$-approximation} if $D_X$ is in $\DD$ and every morphism $D \to X$ with $D\in\DD$ factors through $f$.
 \item $\DD$ is said to be \emph{contravariantly finite} if every object in $\CC$ has a right $\DD$-approximation.
\end{enumerate}
Dually we define a \emph{left $\DD$-approximation} and a \emph{covariantly finite} subcategories.
\begin{enumerate}[resume]
 \item $\DD$ is said to be \emph{functorially finite} if $\CC$ is both contravariantly finite and covariantly finite.
\end{enumerate}
\end{definition}

The Auslander-Buchweitz theory gives a systematic method to provide right and left approximations by a certain subcategory. The following is an exact category version of \cite[Theorem 1.1]{ab} and we present a proof for the convenience of the reader.
\begin{proposition}\label{ab}
Let $\EE$ be an exact category and $\XX$ an extension-closed subcategory of $\EE$. Suppose that $\XX$ has enough injectives $\WW$. Then the following hold.
\begin{enumerate}
\item For any $C \in \widehat{\XX}^n$, there exist conflations
\begin{equation}\label{a1}
\xymatrix{
Y_C \inflr &  X_C \deflr^{f} & C, }
\end{equation}
\vspace{-.5cm}
\begin{equation}\label{a2}
\xymatrix{
C \inflr^{g} & Y^C \deflr & X^C,}
\end{equation}
with $X_C,X^C\in\XX$, $Y_C \in \widehat{\WW}^{n-1}$ and $Y^C \in \widehat{\WW}^n$.
\item If $\EE$ has enough projectives and $\XX$ is a preresolving subcategory of $\EE$, then $f$ is a right $\XX$-approximation for $C$ and $g$ is a left $\widehat{\WW}$-approximation for $C$.
\end{enumerate}
\end{proposition}
\begin{proof} 
(1):
The proof is by induction on $n$. Suppose that $C \in \widehat{\XX}^0 = \XX$. Then $0\infl C \defl C$ gives (\ref{a1}). Since $\XX$ has enough injective objects, we have a conflation $C \infl W \defl C'$ with $W$ in $\WW$ and $C'$ in $\XX$, which is the desired conflation (\ref{a2}).

Now let $n\geq 0$ be an integer and $C$ in $\widehat{\XX}^{n+1}$. By the definition of $\widehat{\XX}^{n+1}$, there exists a conflation $D \infl X \defl C$ such that $D$ is in $\widehat{\XX}^n$ and $X$ is in $\XX$. By the induction hypothesis, we have conflations $Y_D \infl X_D \defl D$ and $D \infl Y^D \defl X^D$ with $X_D,X^D\in\XX$, $Y_D\in\widehat{\WW}^{n-1}$ and $Y^D\in\widehat{\WW}^n$. Then we have the following pushout diagram.
\[
\xymatrix{
D \infld \inflr & X \deflr \infld & C \ar@{=}[d] \\
Y^D \inflr \defld & E \deflr \defld & C \\
X^D \ar@{=}[r] & X^D &
}
\]
Since $\XX$ is closed under extensions, $E$ is in $\XX$ by the middle column, hence the middle row gives (\ref{a1}). Because $\XX$ has enough injective objects, we obtain a conflation $E \infl W \defl F$ in $\XX$ with $W\in\WW$ and $F \in \XX$, which induces the following diagram.
\[
\xymatrix{
Y^D \inflr \ar@{=}[d] & E \deflr \infld & C \infld \\
Y^D \inflr & W \deflr \defld & G \defld \\
 & F \ar@{=}[r] & F
}
\]
Thus $G$ is in $\widehat{\WW}^n$ by the middle row, and the right column gives (\ref{a2}).

(2):
Since $\EE$ has enough projectives and $\XX$ is a preresolving subcategory of $\EE$, it follows that $\Ext^{>0}_\EE(\XX,\WW)$ vanishes. Then it is easy to check $\Ext^{>0}_\EE(\XX,\widehat{\WW})=0$, and by using the long exact sequences of $\Ext$ one can easily show that $f$ and $g$ are approximations.
\end{proof}

\begin{corollary}\label{icchi0}
Let $\EE$ be an exact category with enough projectives and $\XX$ a preresolving subcategory of $\EE$ with enough injectives $\WW$. Suppose that $\widehat{\XX} = \EE$ holds. Then $\XX$ is contravariantly finite and $\XX^\perp = \widehat{\WW}$ is covariantly finite. Moreover we have $\XX \cap \XX^\perp = \WW$ and $\add\XX = {}^\perp\WW = {}^\perp\widehat{\WW}$. If $\widehat{\XX}^n = \EE$, then $\XX^\perp = \widehat{\WW}^n$.
\end{corollary}
\begin{proof}
We first show $\XX \cap \XX^\perp = \WW$. Since $\XX$ is a preresolving subcategory of $\EE$, it follows that $\WW \subset \XX^\perp$ holds.
Let $X \in \XX \cap \XX^\perp$. Since $\XX$ has enough injectives $\WW$, there exists a conflation $X \infl W \defl X'$ with $W$ in $\WW$ and $X'$ in $\XX$. Then this sequence splits because $\Ext_\EE^1(X',X)=0$. Thus $X$ is contained in $\add W$, which implies that $X$ is injective in $\XX$. Therefore $X$ is in $\WW$.

Next we show $\XX^\perp = \widehat{\WW}$. Note that $\widehat{\WW} \subset \XX^\perp$ holds, so it suffices to prove $\XX^\perp \subset \widehat{\WW}$.
Let $C$ be in $\XX^\perp$. By Proposition \ref{ab}, we have a conflation $Y_C \infl X_C \defl C$ with $X_C$ in $\XX$ and $Y_C$ in $\widehat{\WW}$. Since $\XX^\perp$ is clearly closed under extensions, we have $X_C\in \XX \cap \XX^\perp =\WW$. It follows from the definition of $\widehat{\WW}$ that $C$ is in $\widehat{\WW}$.
Note that in case $\widehat{\XX}^n = \EE$, we may assume that $Y_C$ is in $\widehat{\WW}^{n-1}$, thus $C$ is actually in $\widehat{\WW}^n$.

Finally we shall show $\add\XX = {}^\perp\WW = {}^\perp\widehat{\WW}$. Clearly $\add \XX \subset {}^\perp\WW = {}^\perp\widehat{\WW}$ holds. 
Let $C$ be in $^\perp\widehat{\WW}$. Then by Proposition \ref{ab}, we have a conflation $Y_C \infl X_C \defl C$ with $Y_C$ in $\widehat{\WW}$ and $X_C$ in $\XX$. Then $C \in {}^\perp \widehat{\WW}$ implies that this sequence splits, which shows that $C$ is a summand of $X_C$. Consequently, $C$ is in $\add\XX$, which completes the proof.
\end{proof}
Immediately we obtain the following criterion for two preresolving subcategories to be the same up to summands.
\begin{corollary}\label{icchi}
Let $\EE$ be an exact category with enough projectives, and let $\XX$ and $\XX'$ be preresolving subcategories of $\EE$ with enough injectives $\WW$ and $\WW'$ respectively. If $\widehat{\XX} = \widehat{\XX'} = \EE$ and $\add\WW = \add \WW'$ hold, then $\add\XX = \add \XX'$.
\end{corollary}
\begin{proof}
The assertion is clear since in this situation $\add\XX = {}^\perp\WW$ holds by Corollary \ref{icchi0}.  
\end{proof}
For an application to cotilting subcategories we studied in Section 4, we have the following result.
\begin{proposition}\label{abcotilt}
Let $\WW$ be an $n$-cotilting subcategory of $\EE$. Then $\XXX_\WW$ is contravariantly finite. Furthermore we have $\XXX_\WW^\perp =\widehat{\WW} = \widehat{\WW}^n$ and $\XXX_\WW = {}^\perp\WW = {}^\perp\widehat{\WW}$.
\end{proposition}
\begin{proof}
The assertions follow from Corollary \ref{icchi0} (apply to the case $\XX:=\XXX_\WW$), Proposition \ref{xxww} and \ref{cotiltingcond}.
\end{proof}

\section{Constructions of exact structures}
In this appendix, we collect two methods to construct new exact structures from a given one. One is to change exact structures on exact categories, and the other is to give a natural exact structures to quotient categories of exact categories.

\emph{In what follows, we denote by $\EE$ an exact category and by $\CC$ an additive subcategory of $\EE$}. Our aim is to construct a new exact structure in which objects in $\CC$ behave as projective or injective objects.

We remark that similar results were given in \cite{drss} in the case of artin $R$-algebras, based on the theory of relative homological algebra developed by Auslander-Solberg \cite{as1,as2,as3}.
\begin{definition}\label{various}
Let $L \infl M \defl N$ be a conflation in $\EE$.
\begin{enumerate}
\item It is called a \emph{$(\CC,-)$-conflation} if $\EE(C,M)\to \EE(C,N)$ is surjective for all $C$ in $\CC$. In this case $L\infl M$ is called a \emph{$(\CC,-)$-inflation} and $M\defl N$ is called a \emph{$(\CC,-)$-deflation}. 
\item It is called a \emph{$(-,\CC)$-conflation} if $\EE(M,C) \to \EE(L,C)$ is surjective for all $C\in\CC$. 
\item It is called a \emph{$\CC$-conflation} if it is both a $(\CC,-)$-conflation and a $(-,\CC)$-conflation. 
\end{enumerate}
In the obvious way we define the terms \emph{$(-,\CC)$-inflation}, \emph{$(-,\CC)$-deflation}, \emph{$\CC$-inflation} and \emph{$\CC$-deflation}.
\end{definition}

\begin{theorem}\label{modify}
Let $\EE$ be an exact category and $\CC$ an additive subcategory of $\EE$. Then the class of all $(\CC,-)$-conflations (resp. $(-,\CC)$-conflations, $\CC$-conflations) defines a new exact structure on $\EE$.
\end{theorem}
\begin{proof}
We first show that all $(\CC,-)$-conflations defines a new exact structure on $\EE$. It suffices to check Keller's axiom (Ex0), (Ex1), (Ex2) and (Ex2)$^{\op}$ in \cite[Appendix A]{keller}. Note that the class of all $(\CC,-)$-conflations are clearly closed under isomorphisms and (Ex0) ``the identity map of zero object is $(\CC,-)$-deflations'' is trivial. 

(Ex1) \emph{The composition of two $(\CC,-)$-deflations is a $(\CC,-)$-deflation.}

Suppose that $X \defl Y$ and $Y \defl Z$ be $(\CC,-)$-deflations. From the definition, $X \defl Y$ and $Y \defl Z$ are deflations in $\EE$. Thus the composition $X \defl Y \defl Z$ is also a deflation in $\EE$. Then the claim follows since the composition $\EE(\CC,X) \defl \EE(\CC,Y) \defl \EE(\CC,Z)$ is surjective.

(Ex2) \emph{The class of $(\CC,-)$-deflations is stable under pullbacks.}

Suppose that $L \infl M \defl N$ be a $(\CC,-)$-conflation and that $X \to N$ is an arbitrary morphism. Since $\EE$ is an exact category, there exists a pullback diagram
\[
\xymatrix{
L \inflr \ar@{=}[d] & E \ar[d] \deflr & X \ar[d] \\
L \inflr & M \deflr & N
}
\]
where two rows are conflations. We should check the above row is also a $(\CC,-)$-conflation. It is equivalent to say that any morphism $C \to X$ factors through $E \defl X$ for any $C \in \CC$. Let $C \to X$ be a morphism with $C$ in $\CC$. Since $M \defl N$ is a $(\CC,-)$-deflation, the composition $C \to X \to N$ factors through $M \defl N$.
Since the right square is a pullback diagram, there exists a morphism $C \to E$ such that $C \to E \to X$ is equal to $C \to X$. 

(Ex2)$^{\op}$  \emph{The class of $(\CC,-)$-inflations is stable under pushouts.}

Suppose that $L \infl M \defl N$ is a $(\CC,-)$-conflation, and $L \to X$ is an arbitrary morphism. Since $\EE$ is an exact category, There exists a pullback diagram
\[
\xymatrix{
L \inflr \ar[d] & M \ar[d] \deflr & N \ar@{=}[d] \\
X \inflr & E \deflr & N
}
\]
where two rows are conflations. We should show that any morphism $C\to N$ factors through $E\defl N$ for $C\in\CC$. This is easy because $C\to N$ factors through $M \defl N$ and $M \defl N$ factors through $E\defl N$.

Thus we have proved that the class of $(\CC,-)$-conflations defines an exact structure on $\EE$.

Dually the class of $(-,\CC)$-conflations also defines another exact structure on $\EE$. It is easy to check that the intersection of two exact structures gives another exact structure, which implies that the class of $\CC$-conflations also defines an exact structure on $\EE$.
\end{proof}

Next we consider the ideal quotients of an exact category. The following observation gives a natural way to introduce an exact structure to the ideal quotient of $\EE$.
\begin{proposition}\cite[Theorem 3.6]{di}\label{di}
Let $\EE$ be an exact category and $\CC$ an additive subcategory of $\EE$. Denote by $\pi:\EE \defl \EE/[\CC]$ the natural functor. Suppose that every object in $\CC$ is both projective and injective. Then the following are equivalent.
\begin{enumerate}
\item $\EE/[\CC]$ is an exact category whose conflations are precisely the essential images of conflations in $\EE$ under $\pi$.
\item The images of inflations in $\EE$ under $\pi$ are monomorphisms and the images of deflations in $\EE$ under $\pi$ are epimorphisms.
\end{enumerate}
In this case, if moreover $\EE$ has enough projectives $\PP$, then $\EE/[\CC]$ has enough projectives $\add\pi(\PP)$.
\end{proposition}
\begin{proof}
We only prove the last assertion. It is clear from the definition of the exact structure on $\EE/[\CC]$ that the images of objects in $\PP$ under $\pi$ are projective in $\EE/[\CC]$. For any object $X$ in $\EE$, we have a conflation $\Omega X \infl P \defl X$ in $\EE$ with $P$ being projective. Sending it by $\pi$, we obtain a conflation $\u{\Omega X} \infl \u{P} \defl \u{X}$. From this it follows that $\EE/[\CC]$ has enough projectives, and that projective objects are precisely objects in $\add\pi(\PP)$.
\end{proof}

\medskip\noindent
{\bf Acknowledgement.}
The author would like to express his sincere gratitude to his supervisor Osamu Iyama for his patient guidance, numerous suggestions and valuable advice.

\end{document}